\documentclass[reqno,12pt]{amsart}
\usepackage{amsmath,amssymb,latexsym,soul,cite,mathrsfs}
\usepackage{blindtext}
\usepackage{enumitem}
\usepackage{ragged2e}
\usepackage[utf8]{inputenc}
\usepackage{amsthm,amsfonts}
\usepackage{comment}
\usepackage[mathcal]{eucal}
\usepackage{latexsym}
\usepackage[utf8]{inputenc}
\usepackage{bm}
\usepackage[all]{xy}
\usepackage{xcolor}
\usepackage{graphicx}
\usepackage{hyperref}

\usepackage[hyperpageref]{backref}

\newtheorem{theorem}{Theorem}

\newtheorem{theo}{Theorem}[section]

\newtheorem{remark}{Remark}[section]

\newtheorem{lemma}{Lemma}[section]

\newtheorem{prop}{Proposition}[section]

\newtheorem{defi}{Definition}[section]

\newtheorem{cor}{Corollary}[section]

\numberwithin{equation}{section}

\title[SOBOLEV EMBEDDINGS ON SPHERICALLY SYMMETRIC MANIFOLDS]{RADIAL SOBOLEV EMBEDDINGS ON SPHERICALLY SYMMETRIC RIEMANNIAN MANIFOLDS}

\author[J.M.\ do \'O]{Jo\~ao Marcos do \'O}
\author[G. Lu]{Guozhen Lu}
\author[R.~Ponciano]{Raon\'{\i} Ponciano}

\address[Jo\~{a}o Marcos do \'O]{Dep. Mathematics,
	Federal University of Para\'{\i}ba
	\newline\indent
	58051-900, Jo\~ao Pessoa-PB, Brazil}
\email{\href{mailto:jmbo@pq.cnpq.br}{jmbo@pq.cnpq.br}}

\address[Guozhen Lu]{Dep. Mathematics, University of Connecticut
	\newline\indent
	06269, Storrs-CT, United States of America}
\email{\href{mailto:guozhen.lu@uconn.edu}{guozhen.lu@uconn.edu}}

\address[Raon\'i Ponciano]{Dep. of Mathematics,
	Federal University of ABC
	\newline\indent
	09280-560, Santo Andr\'e-SP, Brazil}
\email{\href{mailto:raoni.ponciano@ufabc.edu.br}{raoni.ponciano@ufabc.edu.br}}



\begin{document}

\begin{abstract}
We study Sobolev spaces of radial functions on spherically symmetric Riemannian manifolds. Using geodesic polar coordinates, we give a sharp one-dimensional reduction: a radial function belongs to the Sobolev space on the manifold if and only if its radial representation lies in an associated weighted Sobolev space on an interval, with weights determined explicitly by the metric. This characterization allows us to prove optimal Sobolev-type embeddings for radial functions into weighted Lebesgue spaces on both bounded and unbounded spherically symmetric manifolds. As further consequences, we establish new radial lemmas and decay estimates that capture the precise behaviour of radial Sobolev functions near the origin and at infinity. Our results unify and extend the classical radial embeddings in Euclidean and hyperbolic spaces.
\end{abstract}

\maketitle

\begin{center}
		\footnotesize
		\tableofcontents
	\end{center}

\section{Introduction}
Sobolev embedding theorems are fundamental tools in the analysis of partial differential equations, as they provide a connection between different function spaces and relate regularity properties. In geometric analysis, these embeddings play a central role, since they link the behavior of functions to the geometry. Such connections are essential for deriving a priori estimates, compactness properties, and regularity results for solutions of partial differential equations. They are also a key ingredient in variational methods. Beyond their theoretical importance, Sobolev embeddings have broad applicability across mathematics and the applied sciences, including geometry, physics, and engineering.

\subsection{Weighted Sobolev Embeddings in the Euclidean Setting}

First, let us consider embedding theorems for radial functions defined on Euclidean balls, where no boundary conditions are imposed. Define $L^q_{|x|^\theta}(B_R^{\mathbb R})$ as the space of measurable functions $u\colon B_R^{\mathbb R}\to\mathbb R$ satisfying $\int_{B_R^{\mathbb R}}|u|^q|x|^\theta\mathrm dx<\infty$, where $q\geq1$, $\theta\geq0$, and $B_R^{\mathbb R}\subset \mathbb R^N$ is an open Euclidean ball centered at the origin with radius $R\in(0,\infty]$. The following pioneering theorem by de Figueiredo et al. \cite[Theorem 1.1]{MR2838041} introduced higher order embedding results for Sobolev spaces of radial functions on Euclidean balls, as well as the development of radial lemmas for $u\in W_{\mathbb R,\mathrm{rad}}^{k,p}(B_R^{\mathbb R})$, where $W_{\mathbb R,\mathrm{rad}}^{k,p}(B_R^{\mathbb R})$ denotes the space of radially symmetric functions in the Sobolev space $W^{k,p}_{\mathbb{R}}(B_R^{\mathbb{R}})$.

\begin{theorem}\label{theoDSM}
Assume $R\in(0,\infty)$, $p\geq1$ real, and $k \geq 1$ integer.
\begin{flushleft}
$\mathrm{(1)}$ Every function \( u \in W^{k,p}_{\mathbb{R},\mathrm{rad}}(B_R^{\mathbb{R}}) \) is almost everywhere equal to a function \( U \in C^{k-1}(\overline{B_R^{\mathbb{R}}} \setminus \{0\}) \). Moreover, all partial derivatives of \( U \) of order \( k \) (in the classical sense) exist almost everywhere for \( |x| \in (0, R) \).

$\mathrm{(2)}$ If \( N > kp \), then \( W^{k,p}_{\mathbb{R},\mathrm{rad}}(B_R^{\mathbb{R}}) \) is continuously embedded in \( L^q_{|x|^\theta}(B_R^{\mathbb{R}}) \) for every \( \theta \geq 0 \) and \( 1 \leq q \leq \frac{(\theta + N)p}{N - kp} \).

$\mathrm{(3)}$ If \( N = kp \) and \( p > 1 \), then \( W^{k,p}_{\mathbb{R},\mathrm{rad}}(B_R^{\mathbb{R}}) \) is compactly embedded in \( L^q_{|x|^\theta}(B_R^{\mathbb{R}}) \) for all \( \theta \geq 0 \) and \( 1 \leq q < \infty \).

$\mathrm{(4)}$ If \( N = kp \) and \( p = 1 \), then \( W^{N,1}_{\mathbb{R},\mathrm{rad}}(B_R^{\mathbb{R}}) \) is continuously embedded in \( C(\overline{B_R^{\mathbb{R}}}) \).
\end{flushleft}
\end{theorem}
To prove Theorem \ref{theoDSM}, the following radial lemmas were crucial to the argument. If $N>kp$, then there exists $C>0$ such that for all $u\in W^{k,p}_{\mathbb R,\mathrm{rad}}(B_R^{\mathbb R})$,
\begin{equation*}
|u(x)|\leq C\dfrac{\|u\|_{W^{k,p}(B_R^{\mathbb R})}}{|x|^{\frac{N-kp}p}},\quad\mbox{a.e. } x\in \overline{B_R^{\mathbb R}}\backslash\{0\}.
\end{equation*}
If $N=kp$ and $p>1$, then there exists $C>0$ such that for all $u\in W^{k,p}_{\mathbb R,\mathrm{rad}}(B_R^{\mathbb R})$,
\begin{equation*}
|u(x)|\leq C\|u\|_{W^{k,p}(B_R^{\mathbb R})}\left[\left(\log\frac{R}{|x|}\right)^{\frac{p-1}p}+1\right],\quad\mbox{a.e. } x\in \overline{B_R^{\mathbb R}}\backslash\{0\}.
\end{equation*}

We also mention that the first radial lemmas were established by Ni \cite[Equation (4)]{MR674869} for functions in $H^1_{\mathbb R,\mathrm{rad},0}(B^{\mathbb R}_1)$, and were later extended by Gazzini and Serra \cite[Lemma 2.1]{MR2396523} for functions in $H^1_{\mathbb R,\mathrm{rad}}(B^{\mathbb R}_1)$. Additionally, Dalmasso \cite[Lemme 3.1]{MR1056158} obtained a radial lemma for functions in $H^2_{\mathbb R,\mathrm{rad},0}(B_1^{\mathbb R})$. For the non-Hilbert Sobolev space $W^{1,p}_{\mathbb R,\mathrm{rad}}(B_R^{\mathbb R})$, the corresponding radial lemma was established in \cite[Corollary 2.2.]{zbMATH05796526}.

In the context of the domain being the entire Euclidean space $\mathbb R^N$, our previous work, notably \cite[Theorem 1.1]{MR4819618}, together with the item (1) of Theorem \ref{theoDSM}, yields the following result.
\begin{theorem}\label{theoB}
Assume $p\geq1$ real and $k \geq 1$ integer.
\begin{flushleft}
    $\mathrm{(1)}$ \justifying{Every function $u\in W^{k,p}_{\mathbb R,\mathrm{rad}}(\mathbb R^N)$ is almost everywhere equal to a function $U$ in $C^{k-1}(\mathbb R^N\backslash\{0\})$. Moreover, all partial derivatives of $U$ of order $k$ (in the classical sense) exist almost everywhere for $|x|\in(0,\infty)$.}
    
    \noindent$\mathrm{(2)}$ \justifying{If $N>kp$, then $W^{k,p}_{\mathbb R,\mathrm{rad}}(\mathbb R^N)$ is continuously embedded in $L^q_{|x|^\theta}(\mathbb R^N)$ for every $\theta\geq0$ and $p\leq q\leq \frac{(\theta+N)p}{N-kp}$.}
    
    \noindent$\mathrm{(3)}$ \justifying{If $N=kp$, then $W^{k,p}_{\mathbb R,\mathrm{rad}}(\mathbb R^N)$ is compactly embedded in $L^q_{|x|^\theta}(\mathbb R^N)$ for all $\theta\geq0$ and $p\leq q<\infty$.}
\end{flushleft}
\end{theorem}

In Theorem \ref{theoB}, the estimate $\|u\|_{L^q_{|x|^\theta}(\mathbb R^N)} \leq C\|u\|_ {W^{k,p}(\mathbb{R}^N)}$ was obtained, utilizing the full norm of the Sobolev space $W^{k,p}(\mathbb{R}^N)$. For the estimate $\|u\|_{L^q_{|x|^\theta}(\mathbb R^N)} \leq C \|\nabla^k u\|_{L^p_{|x|^\alpha}(\mathbb R^N)}$, which involves only the norm in $L^p_{|x|^\alpha}(\mathbb{R}^N)$ of the $k$-th gradient, we refer to \cite{zbMATH06376695}. Under the additional assumption of radial symmetry, these two estimates become equivalent, since the norms $\|u\|_{W^{k,p}(\mathbb R^N)}$ and $\|\nabla^ku\|^p_{L^p_{|x|^\alpha}(\mathbb R^N)}$ are equivalent on the radial Sobolev space $W^{k,p}_{\mathbb R,\mathrm{rad}}(\mathbb R^N)$; see \cite[Theorem 1.1]{MR4721763}.

To establish Theorem \ref{theoB}, we relied on \cite[Lemme II.1]{zbMATH03789294} (see also Strauss \cite[Lemma 1]{MR454365}), which provides the following decay property for functions in $W^{1,p}_{\mathbb R,\mathrm{rad}}(\mathbb R^N)$ ($N\geq2$),
\begin{equation}\label{dlrn}
|u(x)|\leq \left(\dfrac{p}{\omega_{N-1}}\right)^{\frac1p}\|u\|_{L^p(\mathbb R^N)}^{\frac{p-1}p}\|\nabla u\|^{\frac1p}_{L^p(\mathbb R^N)}|x|^{-\frac{N-1}p},\quad\mbox{a.e. } x\in\mathbb R^N\backslash\{0\},
\end{equation}
where $\omega_{N-1}$ denotes the $(N-1)$-dimensional volume of the unit sphere $\mathbb S^{N-1}$.

\subsection{Weighted Sobolev Embeddings in the Hyperbolic Setting}

In our recent work \cite{MR4959099}, we studied the Sobolev space of radial functions $W^{k,p}_{\mathbb H,\mathrm{rad}}(B_R^\mathbb H)$, defined on an open ball $B_R^{\mathbb H}=\{x\in\mathbb H^N\colon d(x)<R\}$, where $d(x)$ is the distance between $x$ and $o$ and the functions are radial with respect to the origin $o$. In this context, we allowed $R\in(0,\infty]$, including both finite and infinite radius for the ball. There, we established the Sobolev embedding of this space into $L^q_{\sinh^\theta}(B_R^\mathbb H)$ for $R\in(0,\infty]$, where $L^q_{\sinh^\theta}(B^{\mathbb H}_R)$ is the space of measurable functions $u\colon B_R^{\mathbb H}\to\mathbb R$ satisfying $\int_{B_R^{\mathbb H}}|u|^q\sinh^\theta d(x)\mathrm dx<\infty$. The case $R\in(0,\infty)$ is stated in \cite[Theorem 1.3]{MR4959099} and in the following theorem.
\begin{theorem}\label{theoC}
Assume $R\in(0,\infty)$, $p\geq1$ real, and $k\geq1$ integer.
\begin{flushleft}
    $\mathrm{(1)}$ \justifying{Every function $u\in W^{k,p}_{\mathbb H,\mathrm{rad}}(B_R^\mathbb H)$ is almost everywhere equal to a function $U$ in $C^{k-1}(\overline{B_R^\mathbb H}\backslash\{0\})$. Additionally, all partial derivatives of $U$ of order $k$ (in the classical sense) exist almost everywhere for $d(x)\in(0,R)$.}
    
    \noindent$\mathrm{(2)}$ \justifying{If $N>kp$, then $W^{k,p}_{\mathbb H,\mathrm{rad}}(B_R^{\mathbb H})$ is continuously embedded in $L^q_{\sinh^\theta}(B_R^{\mathbb H})$ for every $\theta\geq0$ and $1\leq q\leq \frac{(\theta+N)p}{N-kp}$.}
    
     \noindent$\mathrm{(3)}$ \justifying{If $N=kp$ and $p>1$, then $W^{k,p}_{\mathbb H,\mathrm{rad}}(B_R^{\mathbb H})$ is compactly embedded in $L^q_{\sinh^\theta}(B_R^{\mathbb H})$ for all $\theta\geq0$ and $1\leq q<\infty$.}
     
     \noindent$\mathrm{(4)}$ \justifying{If $N=kp$ and $p=1$, then $W^{N,1}_{\mathbb H,\mathrm{rad}}(B_R^{\mathbb H})$ is continuously embedded in $C(\overline{B_R^{\mathbb H}})$.}
\end{flushleft}
\end{theorem}

The following radial lemmas were established to prove the previous theorem. If $N>kp$, then there exists $C>0$ such that for all $u\in W^{k,p}_{\mathbb H,\mathrm{rad}}(B_R^\mathbb H)$,

\begin{equation*}
|u(x)|\leq C\dfrac{\|u\|_{W^{k,p}_{\mathbb H}(B_R^{\mathbb H})}}{\sinh^{\frac{N-kp}{p}}d(x)},\quad\mbox{a.e. } x\in\overline{B_R^{\mathbb H}}\backslash\{0\}.
\end{equation*}
If $N=kp$ and $p>1$, then there exists $C>0$ such that for all $u\in W^{k,p}_{\mathbb H,\mathrm{rad}}(B_R^{\mathbb H})$,
\begin{equation*}
|u(x)|\leq C\|u\|_{W^{k,p}_{\mathbb H}(B_R^{\mathbb H})}\left[\left(\log\frac{\tanh(\frac{R}{2})}{\tanh(\frac{d(x)}{2})}\right)^{\frac{p-1}p}+1\right],\quad\mbox{a.e. } x\in \overline{B_R^{\mathbb H}}\backslash\{0\}.
\end{equation*}

In the case where $B_R^{\mathbb H}$ is the entire hyperbolic space $\mathbb H^N$ (that is, $R=\infty$), we also established a Sobolev embedding result, as presented in \cite[Theorem 1.4]{MR4959099}. The theorem is stated as follows.

\begin{theorem}\label{theoD}
Assume $\theta\geq0$, $p\geq1$ , and $k\geq1$ an integer.
\begin{flushleft}
$\mathrm{(1)}$ \justifying{Every function $u\in W^{k,p}_{\mathbb H,\mathrm{rad}}(\mathbb H^N)$ is almost everywhere equal to a function $U$ in $C^{k-1}(\mathbb H^N\backslash\{0\})$. In addition, all partial derivatives of $U$ of order $k$ (in the classical sense) exist almost everywhere for $d(x)\in(0,\infty)$.}

\noindent$\mathrm{(2)}$ If $N>kp$, then the following continuous embedding holds:
\begin{equation*}
W^{k,p}_{\mathbb H,\mathrm{rad}}(\mathbb H^N)\hookrightarrow L^q_{\sinh^\theta}(\mathbb H^N) \quad \text{if} \quad p\leq q\leq \dfrac{(\theta+N)p}{N-kp}=:p^*.
\end{equation*}
Moreover, it is a compact embedding if one of the following two conditions is fulfilled:
\begin{flushleft}
$\mathrm{(i)}$ $\theta=N-1$ and $p<q<p^*$;

\noindent$\mathrm{(ii)}$ $\theta<N-1$ and $p\leq q<p^*$.
\end{flushleft}
\noindent$\mathrm{(3)}$ \justifying{If $N=kp$, then the following continuous embedding holds:}
\begin{equation*}
W^{k,p}_{\mathbb H,\mathrm{rad}}(\mathbb H^N)\hookrightarrow L^q_{\sinh^\theta}(\mathbb H^N)\quad \text{if}\quad p\leq q<\infty.
\end{equation*}
Moreover, it is compact embedding if one of the following two conditions is fulfilled:
\begin{flushleft}
\noindent$\mathrm{(i)}$ $\theta=N-1$ and $q>p$;

\noindent$\mathrm{(ii)}$ $\theta<N-1$ and $q\geq p$.
\end{flushleft}
\end{flushleft}
\end{theorem}

To prove this embedding, it was necessary to develop the following decay lemma. For any $u\in W^{1,p}_{\mathbb H,\mathrm{rad}}(\mathbb H^N)$, the following inequality holds:
\begin{equation}\label{dlhn}
|u(x)|\leq \left(\frac{p}{\omega_{N-1}}\right)^{\frac1p}\|u\|_{L^p(\mathbb H^N)}^{\frac{p-1}p}\|\nabla u\|_{L^p(\mathbb H^N)}^{\frac1p}\sinh^{\frac{1-N}p} d(x),\quad\mbox{a.e. } x\in\mathbb H^N\backslash\{0\}.
\end{equation}

Without the radial symmetry and with $\theta=0$, the borderline case $N=kp$ allows an improvement of the power-type embeddings in item (3) of Theorems \ref{theoDSM}, \ref{theoB}, \ref{theoC}, and \ref{theoD}, where polynomial growth can be replaced by exponential growth. This is given by the Trudinger-Moser type inequalities, which correspond to the critical counterpart of Sobolev embeddings when the limiting exponent becomes infinite. In the Euclidean framework, these inequalities are classical; see, for instance, \cite{Cao,zbMATH03261965, CLZ-CBM, CLZ-JGEA, CLZ-CVPDE, doO, zbMATH03337983,MR2988207, LamLu-JDE, LamLu, LamLuZhang-Revista, MR1646323,MR2400264,zbMATH04099653, RufSani}.

Analogous improvements hold in non-Euclidean geometries. For $N=2$, Mancini and Sandeep \cite{zbMATH05839583} established a sharp Moser-Trudinger inequality on conformal discs.  Mancini, Sandeep, and Tintarev \cite{zbMATH06214509} proved a sharp Moser-Trudinger inequality on hyperbolic space for all $N\geq2$. Lu and Tang \cite{zbMATH06256981} derived both critical and subcritical Trudinger-Moser inequalities on bounded domains and on the whole hyperbolic space, including singular versions. They also obtained sharp inequalities of exact growth in hyperbolic space \cite{zbMATH06579881}. Further sharp Hardy-Trudinger-Moser inequalities on hyperbolic space were obtained by Wang and Ye \cite{zbMATH06029076}, Nguyen \cite{MR3759470, Nguyen-TAMS} and by Liang et al. \cite{zbMATH07247135}.

For higher-order operators, Adams-type inequalities in hyperbolic space were established by Karmakar and Sandeep \cite{zbMATH06536855}, and by Ngô and Nguyen \cite{zbMATH07318489}. Sharp Hardy-Adams inequalities on real hyperbolic spaces were esablished by Li, Lu, and Yang \cite{zbMATH06776234, zbMATH06898920, zbMATH07189073} using the ideas of applying Helgason-Fourier analysis on symmetruc spaces to prove sharp geometric inequalities developed by Lu and Yang \cite{LuYang-AJM}.
Extensions to other rank-one symmetric spaces were later obtained by Lu and Yang in complex hyperbolic space \cite{zbMATH07557162}, and by Flynn, Lu, and Yang in quaternionic and octonionic hyperbolic spaces \cite{zbMATH07835927}. Finally, optimal Trudinger-Moser and Adams type inequalities on more general Riemannian manifolds can also be found, for example, in \cite{zbMATH00490628, ChenW, LiLu-arXiv, LiLuZhu-ANS, YXLi1, YXLi2, YXLi3, YXLi4, Yang-JFA,  zbMATH00447008, YangQH, zbMATH07384223}, just to name a few.

\subsection{Spherically Symmetric Riemannian Manifolds} The Spherically symmetric Riemannian manifolds were crucial in the geometric analysis associated with the positive (ADM) mass theorem in general relativity. Early proofs of the theorem in the spherically symmetric setting were obtained by Jang \cite{MR403545}, Leibovitz \cite{MR1552569}, and Misner \cite{misner1964relativistic}, where the notion of mass is given by the ADM energy-momentum introduced by Arnowitt, Deser, and Misner \cite{MR127946}. The theorem was later established in full generality in three dimensions through independent works of Schoen and Yau \cite{MR526976} and Witten \cite{MR626707}; see also \cite{MR849427,MR994021} for related developments. More recently, Lee and Sormani \cite{MR3176604} investigated the stability of the positive mass theorem for a class of $N$-dimensional spherically symmetric Riemannian manifolds. Their framework includes several physically relevant examples, involving Schwarzschild manifolds and classical spherically symmetric gravitational wells. We also remark that their analysis is restricted to asymptotically Euclidean geometries. However, in \cite{MR3691759}, Sakovich and Sormani considered the analogous problem in the setting of asymptotically hyperbolic spaces.

We begin by motivating the concept of spherically symmetric Riemannian manifolds (see Definition \ref{defissm}). Our focus is on studying radial functions on an $N$-dimensional ($N\geq2$) Riemannian manifold $(M,g)$. Let $o\in M$ (referred to as the origin) and denote $d(x)$ as the distance between $x$ and $o$. A function $u\colon M\to\mathbb R$ is called \textit{radial} if depends only on $d(x)$, i.e.,
\begin{equation*}
d(x)=d(y)\Longrightarrow u(x)=u(y).
\end{equation*}
This implies the existence of a function $v\colon[0,\sup_{x\in M}d(x))\to\mathbb R$ such that $u(x)=v(d(x))$. Assuming $u$ is integrable, we aim to make use of an identity analogous to the following one, which holds in the Euclidean space:
\begin{equation}\label{layercakeRn}
\int_{B_R(o)}u(x)\mathrm dx=\omega_{N-1}\int_0^Rv(t)t^{N-1}\mathrm dt.
\end{equation}

For a general Riemannian manifold, we can derive a similar expression, at least locally. Let $B_R(o)$ be a normal ball of radius $R\in(0,\infty]$. Then, the exponential map

\begin{equation*}
\exp_o\colon B_R(o)\subset T_oM\to M
\end{equation*}
is a diffeomorphism onto its image. Under the induced metric $\exp_o^*g$ on $B_R(o)$, the map $\exp_o$ becomes an isometry. By Gauss' Lemma (see \cite[2.93 Gauss Lemma]{MR2088027}), the metric $g$ can be expressed in polar coordinates as:
\begin{equation*}
g=\mathrm dr^2+h_{ij}(r,\theta)\mathrm d\theta^i\mathrm d\theta^j,
\end{equation*}
where $\mathrm dr$ is the 1-form associated with the radial vector field and $h_{ij}(r,\theta)\mathrm d\theta^i\mathrm d\theta^j$ represents the angular components of the metric. Equivalently, in matrix form:
\begin{equation*}
g=\begin{bmatrix}
1 & 0\quad\cdots\quad0 \\
\begin{matrix}
0\\
\vdots\\
0
\end{matrix} & \boxed{
\begin{matrix}
 &  &  \\
 & h_{ij}(r,\theta) &  \\
 &  & 
\end{matrix}}
\end{bmatrix}
\end{equation*}

The volume element associated with $g$ is given by 
\begin{equation*}
\mathrm dV_g=\sqrt{\det(h_{ij}(r,\theta))}\mathrm dr\mathrm d\theta,
\end{equation*}
where $\mathrm d\theta=\mathrm d\theta^1\wedge\cdots\wedge\mathrm d\theta^{N-1}$ is the volume element on the unit sphere $\mathbb S^{N-1}$. Using this, we obtain the following generalization of \eqref{layercakeRn} for any Riemannian manifold $(M,g)$:
\begin{equation}\label{layercakehard}
\int_{B_R(o)}u(x)\mathrm dV_g=\int_0^Rv(t)\int_{\mathbb S^{N-1}}\sqrt{\det(h_{ij}(t,\theta))}\mathrm d\theta\mathrm dt.
\end{equation}
However, the integral over $\mathbb S^{N-1}$ in \eqref{layercakehard} lacks the desired properties needed for our arguments. To address this, we assume that $h_{ij}(r,\theta)=\phi^2(r)\delta_{ij}$, where $\delta_{ij}$ is the Kronecker delta. This assumption implies that the angular components are independent of $\theta$. Geometrically, this corresponds to assuming that the metric is ``radial" around $o$. With this simplification, we get
\begin{equation}\label{layercake}
\int_{B_R(o)}u(x)\mathrm dV_g=\omega_{N-1}\int_0^Rv(t)\phi^{N-1}(t)\mathrm dt.
\end{equation}
With these considerations, we now adopt the notion of \textit{model manifold} as given in \cite[Definition 3.21]{MR2569498}. Throughout this work, however, we will refer to such manifolds as \textit{spherically symmetric Riemannian manifolds}. We emphasize that, in contrast with several standard definitions in the literature, our framework does not require the manifold to possess a pole. In particular, the exponential map at the origin need not be defined on the entire tangent space, allowing the manifold to be bounded in this sense.
\begin{defi}\label{defissm}
An $N$-dimensional Riemannian manifold ($M,g)$ is called a spherically symmetric Riemannian manifold if the following two conditions are satisfied:
\begin{enumerate}
\item There is a chart on $M$ that covers all $M$, and the image of this chart in $\mathbb{R}^N$ is a ball
\begin{equation*}
B_{R}=\left\{x \in \mathbb{R}^N:|x|<R\right\}
\end{equation*}
of radius $R\in(0,+\infty]$. In particular, if $R=\infty$, then $B_{R}=\mathbb{R}^N$.
\item The metric $g$ in the polar coordinates $(r, \theta)$ with respect to the origin $o$ in the above chart has the form
\begin{equation}\label{gSSRM}
g=\mathrm d r^2+\phi^2(r)\widetilde g,
\end{equation}
where $\widetilde g$ is the standard metric in $\mathbb S^{N-1}$ and $\phi$ is a nonnegative smooth function on $[0,R)$, positive on $(0,R)$, satisfying
\begin{equation}\label{fsmooth}
\phi'(0)=1\mbox{ and }\phi^{(i)}(0)=0\mbox{ for all }i\geq0\mbox{ even}.
\end{equation}
\end{enumerate}
\end{defi}
Throughout the paper, $M$ always denotes a spherically symmetric Riemannian manifold. Moreover, we denote by $o$ the \textit{origin}, which corresponds to $0\in\mathbb R^N$ via the chart. In particular, $M\backslash\{o\}$ is the warped product $(0,R)\times_\phi\mathbb S^{N-1}$. For a background on warped products, we refer to \cite[Section 3]{MR3699316}. For notation convenience, we shall simply write
\begin{equation*}
M\backslash\{o\}=(0,R)\times\mathbb S^{N-1}\mbox{ and }M=B_R,
\end{equation*}
with the understanding that these identifications are made via the chart and polar coordinates. Consequently, any point $x\in M\backslash\{o\}$ is uniquely represented as $(r,\theta)\in(0,R)\times\mathbb S^{N-1}$.

The condition \eqref{fsmooth} is necessary and sufficient to extend the metric \eqref{gSSRM} smoothly to the origin. For simplicity, we assume smoothness throughout the paper. However, to work in $W^{k,p}(M)$, it is sufficient that the manifold $M$ is of class $C^k$. In this case, the optimal regularity assumption on $\phi$ is that the functions
\begin{equation*}
x\mapsto\dfrac{\phi^2(|x|)}{|x|^2}\mbox{ and }x\mapsto\dfrac{|x|^2-\phi^2(|x|)}{|x|^4},
\end{equation*}
defined on $B_R\backslash\{0\}$, extend as $C^k$ functions at the origin. For a detailed explanation of this equivalence, we refer \cite[Section 1.4.4]{MR3469435}. In some results, we additionally assume $R<\infty$ and that the limit $\lim_{r\to R}\phi^{(j)}(r)$ exists for every $j\geq0$. These assumptions are required to extend the metric smoothly up to $r=R$ and imply that $M$ can be identified with the interior of a compact manifold with boundary.

For further details on this type of space, we recommend consulting \cite{MR521983}, \cite[Section 3.10]{MR2569498}, \cite[Sections 4.2 and 4.3]{MR3469435}, and \cite[Sections 11–14 in Chapter VII]{MR1013365}. By solving the ODE presented in \cite[Equation (7.25)']{MR1013365}, we obtain the following examples of spherically symmetric Riemannian manifolds:
\begin{itemize}
\item $M=\mathbb R^N$ with $R=\infty$ and $\phi(r)=r$;
\item $M=\mathbb S^N\backslash\{-o\}$ with $R=\pi$ and $\phi(r)=\sin r$;
\item $M=\mathbb H^N$ with $R=\infty$ and $\phi(r)=\sinh r$.
\end{itemize}
An additional example of a spherically symmetric Riemannian manifold is a surface of revolution, as discussed in \cite[Example 1.4.4]{MR3469435}.

\subsection{Main Results}

We say that $u\colon M\to\mathbb R$ is \textit{radial} if there exists $v\colon(0,R)\to\mathbb R$ such that $u(x)=v(r)$, for all $x=(r,\theta)\in M\backslash\{o\}=(0,R)\times\mathbb S^{N-1}$. Define $W^{k,p}((0,R),\phi^{N-1})$ the radial weighted Sobolev space given by
\begin{equation*}
\left\{v\colon(0,R)\to\mathbb R\mbox{ has }k\mbox{ weak derivatives}\colon \int_0^R|v^{(j)}(t)|^p\phi^{N-1}(t)\mathrm dt<\infty,\ \forall j=0,\ldots,k\right\}.
\end{equation*}
Equipped with the norm

\begin{equation*}
\|v\|_{W^{k,p}_{\phi^{N-1}}}=\left(\int_0^R|v^{(j)}(t)|^p\phi^{N-1}(t)\mathrm dt\right)^{\frac1p},
\end{equation*}
the space $W^{k,p}((0,R),\phi^{N-1})$ is a Banach space.

Our first main theorem is a technical yet crucial result for establishing the desired embeddings. The statement itself is quite simple. However, the proof is not straightforward in the Euclidean and hyperbolic settings. In fact, the Euclidean proof required approximately two pages (see \cite[Proof of Theorem 2.2]{MR2838041}), while the hyperbolic case used about seven pages (see \cite[Proof of Theorem 1.1]{MR4959099}). By contrast, our new argument is contained in roughly one page. The main simplification arises from our interpretation of the spherically symmetric Riemannian manifold via geodesic polar coordinates, rather than the Cartesian coordinates employed in previous works. This approach is both simpler and more general.

\begin{theo}\label{theouv}
Let $u\in W^{k,p}_{\mathrm{rad}}(M)$, then $v\in W^{k,p}((0,R),\phi^{N-1})$. Moreover, for a.e. $x=(r,\theta)\in M$
\begin{equation}\label{d456}
|\nabla^j u(x)|_g\geq |v^{(j)}(r)|,\quad \forall j=0,\ldots,k.
\end{equation}
\end{theo}

This theorem is significant only in the context of higher-order derivatives. Indeed, for a general Riemannian manifold $M$ (not necessarily spherically symmetric), it holds that $|\nabla d(x)|_g=1$ for all $x\in M\backslash(\{o\}\cup\mathrm{Cut}(o))$, where $\mathrm{Cut}(o)$ denotes the cut locus of $o$. Consequently,
\begin{equation*}
|\nabla u(x)|_g=|v'(d(x))|,\quad\mbox{a.e. }x\in M\backslash(\{o\}\cup\mathrm{Cut}(o)).
\end{equation*}
Therefore, in the case $k=1$, the statement of the theorem is immediate and does not provide any additional information. Remarks concerning the higher-order derivative setting on a general $N$-dimensional Riemannian manifold are provided in the final section, see Section \ref{sec7}.

Motivated by Theorem \ref{theouv}, a natural question arises: if $u$ belongs to $W^{k,p}_{\mathrm{rad}}(M)$, does $v$ belong to $W^{k,p}((0,R),\phi^{N-1})$ and vice-versa? Our next main theorem addresses this question.
\begin{theo}\label{theo10}
For each $R\in(0,\infty]$, $p\geq1$, and $k\geq1$ integer, we have
\begin{flushleft}
    $\mathrm{(1)}$ $W^{k,p}_{\mathrm{rad}}(M)\hookrightarrow W^{k,p}((0,R),\phi^{N-1})$;
    
    \noindent$\mathrm{(2)}$ $W^{1,p}_{\mathrm{rad}}(M)\equiv W^{1,p}((0,R),\phi^{N-1})$;
    
    \noindent$\mathrm{(3)}$ Assume that $R\in(0,\infty)$ and $k\geq2$. Suppose that
    \begin{equation*}
    0<\liminf_{r\to R}\phi(r)\leq\limsup_{r\to R}\phi(r)<\infty,
    \end{equation*}
    and that $|\phi'(r)|$ as well as $|\frac{\mathrm d^i}{\mathrm dr^i}(\phi''(r)\phi(r))|$ are bounded on $r\in(0,R)$ for all $i=0,\ldots,k-3$. Then the following equivalence holds:
    \begin{equation*}
        W^{k,p}_{\mathrm{rad}}(M)\equiv W^{k,p}((0,R),\phi^{N-1})\Leftrightarrow N>(k-1)p.
    \end{equation*}
\end{flushleft}
\end{theo}

Generalizing Theorem \ref{theoDSM} for bounded balls in $\mathbb R^N$ and Theorem \ref{theoC} for bounded ball in $\mathbb H^N$, we establish the embedding for $R<\infty$ of $W^{k,p}_{\mathrm{rad}}(M)$ into $L^q_{\phi^\theta}(M)$ with our next main result.
\begin{theo}\label{theo11}
Let $R\in(0,\infty)$, $\theta\geq0$, $p\geq1$, and $k\geq1$. Assume that the limit $\lim_{r\to R}\phi^{(j)}(r)\in(0,\infty)$ exists for every $j\geq0$.
\begin{flushleft}
    $\mathrm{(1)}$ \justifying{Every function $u\in W^{k,p}_{\mathrm{rad}}(M)$ is almost everywhere equal to a function $U$ in $C^{k-1}((0,R]\times \mathbb S^{N-1})$. Additionally, all partial derivatives of $U$ of order $k$ (in the classical sense) exist almost everywhere for $x\in M$.}
    
    \noindent$\mathrm{(2)}$ \justifying{If $N>kp$, then the following continous embedding holds:
    \begin{equation*}
    W^{k,p}_{\mathrm{rad}}(M)\hookrightarrow L^q_{\phi^\theta}(M)\quad\mbox{if}\quad1\leq q\leq p^*_\theta:=\frac{(\theta+N)p}{N-kp}
    \end{equation*}}
    
     \noindent$\mathrm{(3)}$ \justifying{If $N=kp$ and $p>1$, then the following compact embedding holds:
     \begin{equation*}
     W^{k,p}_{\mathrm{rad}}(M)\hookrightarrow L^q_{\phi^\theta}(M)\quad\mbox{if}\quad 1\leq q<\infty.
     \end{equation*}}
     
     \noindent$\mathrm{(4)}$ \justifying{If $N=kp$ and $p=1$, then $W^{N,1}_{\mathrm{rad}}(M)$ is continuously embedded in $C(\overline{B_R})$.}
\end{flushleft}
\end{theo}

\begin{remark}\label{remarkct}
    Assuming no weights in the Lebesgue space ($\theta=0$), the previous result implies that $W^{k,p}_{\mathrm{rad}}(M) \hookrightarrow L^q(M)$ for every $1 \leq q \leq \frac{Np}{N-kp}$. This follows from the classical Sobolev embedding theorem without the need for radial symmetry (see \cite[Theorem 10.1]{MR1688256}).
\end{remark}

\begin{remark}
For an embbeding version from $W^{k,p}((0,R),\phi^{N-1})$ to $L^q_{\phi^{\theta}}(0,R)$, we refer the reader to Theorem \ref{theoweighted}.
\end{remark}

We establish the following radial lemmas (see Proposition \ref{proprls} and \ref{propworst}) to prove Theorem \ref{theo11}, assuming some suitable conditions on $\phi$. If $N>kp$, then there exists $C>0$ such that for all $u\in W^{k,p}_{\mathrm{rad}}(M)$,
\begin{equation}\label{rls}
|u(x)|\leq C\dfrac{\|u\|_{W^{k,p}(M)}}{\phi(r)^\frac{N-kp}p},\quad\mbox{a.e. }x=(r,\theta)\in (0,R]\times \mathbb S^{N-1}.
\end{equation}
Equivalently, it also holds that
\begin{equation}\label{asfkj111}
|u(x)|\leq C\dfrac{\|u\|_{W^{k,p}(M)}}{r^{\frac{N-kp}{p}}},\quad\mbox{a.e. }x=(r,\theta)\in(0,R]\times \mathbb S^{N-1}.
\end{equation}
If $N=kp$ and $p>1$, then there exists $C>0$ such that for all $u\in W^{k,p}_{\mathrm{rad}}(M)$,
\begin{equation}
|u(x)|\leq C\|u\|_{W^{k,p}(M)}\left[\left(\log\frac{R}{r}\right)^{\frac{p-1}p}+1\right],\quad\mbox{a.e. }x=(r,\theta)\in (0,R]\times \mathbb S^{N-1}.\label{rlsc}
\end{equation}
These radial lemmas \eqref{asfkj111} and \eqref{rlsc} were expected because in the Euclidean case, it was proved the same growth type (see \cite[Equations (1.1) and (1.2)]{MR2838041}). The reason is that the vital part of these radial lemmas is the estimate near the origin, where any Riemannian manifold is similar to the Euclidean space.

\begin{cor}\label{corcompact}
Assume the hypotheses of Theorem \ref{theo10}, and suppose that $N>kp$ and $\theta\geq0$. If $1\leq q<p^*_\theta$, then $W^{k,p}_{\mathrm{rad}}(M)\hookrightarrow L^q_{\phi^{\theta}}(M)$ is compact.
\end{cor}

We now turn to the case where $R=\infty$. As in previous works (see \eqref{dlrn} and \eqref{dlhn}), given $u\in W^{1,p}_{\mathrm{rad}}(M)$, it is necessary to obtain an asymptotic decay estimate as $d(x)\to\infty$. Here, we restrict our analysis to those manifolds which the infimum
\begin{equation}\label{Cphi}
C_{\phi}:=\inf_{0<r\leq t}\dfrac{\phi(t)}{\phi(r)}
\end{equation}
is strictly positive. When $\phi$ is nondecreasing, as in the cases of $M=\mathbb R^N$ and $M=\mathbb H^N$, we have $C_\phi=1$. Under this assumption, we establish in Lemma \ref{decaylemma} the following decay estimate
\begin{equation}\label{aksfn}
|u(x)|\leq \left(\frac{p}{C^{N-1}_\phi\omega_{N-1}}\right)^{\frac1p}\|u\|_{L^p(M)}^{\frac{p-1}p}\|\nabla u\|_{L^p(M)}^{\frac1p}\phi(r)^{\frac{1-N}p},\quad\mbox{a.e. }x=(r,\theta)\in (0,\infty)\times\mathbb S^{N-1}.
\end{equation}
We emphasize that, in contrast with the case $R<\infty$ (see \eqref{asfkj111}), the right-hand side of \eqref{aksfn} explicitly depends on $\phi$. This indicates the fact that the volume growth of the manifold $M$ is governed by the warping function $\phi$. In particular, faster expasion of $M$ (corresponding to faster growth of $\phi$) yields stronger decay as $r=d(x)\to\infty$, while slower expansion leads to weaker decay. Consequently, with this decay lemma we establish the Sobolev embedding theorem on the entire manifold $M\backslash\{o\}=(0,\infty)\times\mathbb S^{N-1}$. We state this result next.

\begin{theo}\label{theo12}
Let $R=\infty$, $\theta\geq0$, $p\geq1$, and $k\geq1$. Assume that $C_\phi>0$.
\begin{flushleft}
$\mathrm{(1)}$ \justifying{Every function $u\in W^{k,p}_{\mathrm{rad}}(M)$ is almost everywhere equal to a function $U$ in $C^{k-1}(M\backslash\{o\})$. In addition, all partial derivatives of $U$ of order $k$ (in the classical sense) exist almost everywhere for $x\in M$.}

\noindent$\mathrm{(2)}$ If $N>kp$, then the following continuous embedding holds:
\begin{equation*}
W^{k,p}_{\mathrm{rad}}(M)\hookrightarrow L^q_{\phi^\theta}(M) \quad \text{if} \quad p\leq q\leq p^*_\theta.
\end{equation*}
Moreover, the embedding is compact if $\lim_{r\to\infty}\phi(r)=\infty$, $q<p_\theta^*$, and $p\theta<q(N-1)$.

\noindent$\mathrm{(3)}$ \justifying{If $N=kp$, then the following continuous embedding holds:}
\begin{equation*}
W^{k,p}_{\mathrm{rad}}(M)\hookrightarrow L^q_{\phi^\theta}(M)\quad \text{if}\quad p\leq q<\infty.
\end{equation*}
Moreover, the embedding is compact if $\lim_{r\to\infty}\phi(r)=\infty$ and $p\theta<q(N-1)$.
\end{flushleft}
\end{theo}
\begin{remark}
Similarly to Remark \ref{remarkct}, for $\theta=0$, the previous result implies that $W^{k,p}_{\mathrm{rad}}(M) \hookrightarrow L^q(M)$ for every $p \leq q \leq \frac{Np}{N-kp}$. This follows from the classical Sobolev embedding theorem without the need for radial symmetry (see \cite[Theorem 3.2]{MR1688256}).
\end{remark}

\begin{remark}

For an embedding version from $W^{k,p}((0,\infty),\phi^{N-1})$ to $L^q_{\phi^\theta}(0,\infty)$, we refer to Theorem \ref{theo122}.
\end{remark}

\subsection{Outline of the paper}

The rest of the paper is divided as follows. In Section \ref{sec2}, we collect some preliminary background on $k$-covariant derivatives and Sobolev spaces on Riemannian manifolds. Section \ref{sec3} is devoted to the proof of Theorem \ref{theouv}. Section \ref{sec4} focuses on characterizing when the condition $u\in W^{k,p}_{\mathrm{rad}}(M)$ is equivalent to $v\in W^{k,p}((0,R),\phi^{N-1})$ and proving Theorem \ref{theo10}. Section \ref{sec5} contains the Sobolev embedding results on bounded spherically symmetric Riemannian manifolds, that is, when $R<\infty$. In this section, we establish suitable radial lemmas and use them to prove Theorem \ref{theo11}. Finally, Section \ref{sec6} addresses the unbounded case $R=\infty$, where we prove Theorem \ref{theo12} by proving the corresponding Sobolev embeddings together with appropriate decay lemmas.

\section{Preliminaries}\label{sec2}

Let us assume $M$ being a $N$-dimensional spherically symmetric Riemannian manifold. Denote $\theta_2,\ldots,\theta_{N}$ the spherical coordinates on $\mathbb S^{N-1}$. Denoting $\theta_1=r$ the radial coordinate of $M\backslash\{o\}$, we find that $\theta_1,\ldots,\theta_N$ are the coordinates of $M\backslash\{o\}$. If $\widetilde g$ denotes the metric of $\mathbb S^{N-1}$ and $\widetilde \Gamma_{ij}^k$ denotes the Christoffel symbols of $\mathbb S^{N-1}$ under the spherical coordinates $\theta_2,\ldots,\theta_N$, then
\begin{equation*}
g_{ij}=\left\{\begin{array}{ll}
\phi^2(r)\widetilde g_{ij},&\mbox{if }2\leq i,j\leq N,\\
\delta_{ij},&\mbox{otherwise},
\end{array}\right.\mbox{ and }g^{ij}=\left\{\begin{array}{ll}
\phi^{-2}(r)\widetilde g^{ij},&\mbox{if }2\leq i,j\leq N,\\
\delta^{ij},&\mbox{otherwise}.
\end{array}\right.
\end{equation*}
The explicit form of the spherical coordinates is not relevant for our purposes; however, we will need the fact that $\widetilde g_{ij}=\widetilde g^{ij}=0$ for $i\neq j$. Let $2\leq i,j,k\leq N$. Doing the calculations,
\begin{equation}\label{christoffel}
\Gamma_{ij}^k=\widetilde \Gamma_{ij}^k,\ \Gamma_{i1}^k=\Gamma_{1i}^k=\delta_{ik}\frac{\phi'(r)}{\phi(r)},\ \Gamma_{ij}^1=-\phi(r)\phi'(r)\widetilde g_{ij}\mbox{ and }\Gamma_{11}^k=\Gamma_{i1}^1=\Gamma_{11}^1=0.
\end{equation}

Let us introduce some definitions contained in \cite{MR1688256} to define the Sobolev spaces on Riemannian manifolds. We define the space $T_k(T_xM)$ of all $k$-covariant tensor
\begin{equation*}
\eta\colon \underbrace{T_xM\times\cdots\times T_xM}_\text{$k$-times}\rightarrow\mathbb R.
\end{equation*}
Given a chart of $M$ at $x$, the set
\begin{equation*}
    \left\{dx_{i_1}\otimes\cdots\otimes dx_{i_k}\right\}_{i_1\ldots i_k}
\end{equation*}
is a base of $T_k(T_xM)$, where $\otimes$ is the tensor product. We denote the components of a $k$-covariant tensor $\eta$ by $\eta_{i_1\ldots i_k}$. Then
\begin{equation*}
\eta=\sum_{i_1\ldots i_k=1}^N\eta_{i_1\ldots i_k}dx_{i_1}\otimes\cdots\otimes dx_{i_k}.
\end{equation*}
A map $\eta\colon M\to \amalg_{x\in M}T_k(T_xM)$ is said to be a $k$-covariant tensor field if $\eta(x)\in T_k(T_xM)$ for all $x\in M$. We said that $\eta$ is of class $C^{j}$ if it is of class $C^j$ from the manifold $M$ to the manifold $\amalg_{x\in M}T_k(T_xM)$ (equivalently, $\eta_{i_1\ldots i_k}\colon M\to\mathbb R$ is of class $C^j$ for all $i_1,\ldots,i_k$). The metric $g$ induces a metric in the space of the $k$-covariant tensor field in the following way:
\begin{equation*}
\langle \eta,\nu\rangle_g:=\sum_{i_1\ldots i_kj_1,\ldots j_k}g^{i_1j_1}\cdots g^{i_kj_k}\eta_{i_1\ldots i_k}\nu_{j_1\ldots j_k},
\end{equation*}
where $\eta$ and $\nu$ are $k$-covariant tensor fields.

Given $\eta$ a $k$-covariant tensor field of class $C^{j+1}$, we define the covariant derivative $\nabla\eta$ as a $(k+1)$-covariant tensor field of class $C^j$ whose components are given by
\begin{equation*}
\left(\nabla \eta\right)(x)_{i_1\ldots i_{k+1}}=\dfrac{\partial \eta_{i_2\ldots i_{k+1}}}{\partial x_{i_1}}(x)-\sum_{\ell=2}^{k+1}\sum_{\alpha=1}^N\Gamma^\alpha_{i_1i_\ell}\eta(x)_{i_2\ldots i_{\ell-1}\alpha i_{\ell+1}\ldots i_{k+1}}.
\end{equation*}

For a function $u\colon M\to\mathbb R$ smooth, we denote by $\nabla^ku$ the $k$-covariant derivative of $u$, and $|\nabla^ku|_g$ the norm of $\nabla^ku$ defined by
\begin{equation}\label{normkcov}
|\nabla^ku|_g^2=\sum_{i_1\ldots i_kj_1\ldots j_k=1}^Ng^{i_1j_1}\cdots g^{i_kj_k}\left(\nabla^ku\right)_{i_1\ldots i_k}\left(\nabla^ku\right)_{j_1\ldots j_k}.
\end{equation}

Given $k$ a positive integer and $p\geq1$ real number, we define
\begin{equation*}
\mathfrak C^{k,p}(M)=\left\{u\in C^\infty(M)\colon\int_M|\nabla^ju|_g^p\mathrm dV_g<\infty,\quad\forall j=0,\ldots,k\right\}
\end{equation*}
and $W^{k,p}(M)$ the Sobolev space on a Riemannian manifold by the completion of $\mathfrak C^{k,p}(M)$ under the norm $\|\cdot\|_{W^{k,p}}$, where
\begin{equation*}
    \|u\|_{W^{k,p}(M)}=\sum_{j=0}^k\left(\int_M|\nabla^ju|_g^p\mathrm dV_g\right)^{\frac1p},\quad\mbox{for }u\in \mathfrak C^{k,p}(M).
\end{equation*}

The space $W^{k,p}(M)$ can be interpreted as the space of all functions $u\in L^p(M)$ such that there exists $(u_n)$ Cauchy sequence in $\mathfrak C^{k,p}(M)$ with $u_n\to u$ in $L^p(M)$. Therefore, we define $\nabla^ju$ as the $L^p$-limit of the sequence $(\nabla^ju_n)$ and the norm $\|\cdot\|_{W^{k,p}}$ is defined with the same expression as above.

\section{Proof of Theorem \ref{theouv}}\label{sec3}

In this section, we present the proof of our first main theorem. We emphasize that, in the middle of the proof, we obtain the simple relation
\begin{equation*}
v^{(k)}(r)=\left(\nabla^k u\right)_{1\cdots 1}(x),\quad\mbox{a.e. } x=(r,\theta)\in M\backslash\{o\}
\end{equation*}
as stated in \eqref{eqc456}. For comparison, in the Euclidean setting, the corresponding expression (see \cite[Equation (2.3)]{MR2838041}) takes the form
\begin{equation*}
v^{(k)}\left(d(x)\right)=\sum_{i_1\ldots i_j=1}^N(\nabla^ju)_{i_1\cdots i_k}(x)\dfrac{x_{i_1}\cdots x_{i_j}}{|x|^j},\quad\mbox{a.e. }x\in B_R^{\mathbb R}.
\end{equation*}
On the other hand, in the hyperbolic setting, it was shown in \cite[Lemma 3.2]{MR4959099} that
\begin{equation*}
v^{(j)}(d(x))=\left(\frac{1-|x|^2}2\right)^j\sum_{i_1\ldots i_j=1}^N(\nabla^ju)_{i_1\cdots i_k}(x)\dfrac{x_{i_1}\cdots x_{i_j}}{|x|^j}\quad\mbox{a.e. }x\in B_R^{\mathbb H}.
\end{equation*}
As these formulas illustrate, the additional complexity in the Euclidean and hyperbolic cases arises from the use of Cartesian coordinates. In contrast, in the present work, we apply geodesic polar coordinates, which are more suitable for radial functions and lead to a simpler formula for higher-order derivatives, which holds on a more general class of Riemannian manifolds.

\begin{proof}[Proof of Theorem \ref{theouv}]
We claim that $v$ has $k$ weak derivatives and
\begin{equation}\label{eqc456}
    v^{(k)}(r)=\left(\nabla^k u\right)_{1\cdots 1}(x),\quad\mbox{a.e. } x=(r,\theta)\in M\backslash\{o\}.
\end{equation}
First of all, let us show \eqref{eqc456} for $u$ smooth. Using $\Gamma_{11}^k=\Gamma_{11}^1=0$ and the recursive definition of $k$-covariant derivative, we have, for any $x=(r,\theta)\in M$,
\begin{align}
\left(\nabla^{k}u\right)_{1\cdots 1}(x)&=\left(\nabla_1\nabla^{k-1}u\right)_{1\cdots 1}(x)\nonumber\\
&=\dfrac{\partial(\nabla^{k-1}u)_{1\cdots 1}}{\partial r}(x)-\sum_{\ell=1}^{k-1}\sum_{\alpha=1}^N\Gamma_{11}^\alpha(x)\left(\nabla^{k-1}u\right)_{1\cdots 1\alpha 1\cdots 1}(x)\nonumber\\ 
&=\dfrac{\partial\left(\nabla^{k-1}u\right)_{1\cdots 1}}{\partial r}(x)=\cdots=\dfrac{\partial^ku}{\partial r^k}(x)=v^{(k)}(r),\label{eq456}
\end{align}
where in the term $\left(\nabla^{k-1}u\right)_{1\cdots 1\alpha 1\cdots 1}$ has $\alpha$ on the position $\ell$. This concludes \eqref{eqc456} for smooth function. 

Before we show \eqref{eqc456} for $u\in W^{k,p}_{\mathrm{rad}}(M)$, we claim that given $w\in C^\infty_0(M)$ with $\mathrm{supp}(w)\subset M\backslash\{o\}$, we have
\begin{equation}\label{eqdivtheo}
    \int_{M}u(x)\dfrac{\partial w}{\partial r}(x)\dfrac{1}{\phi^{N-1}(r)}\mathrm dV_g=-\int_{M}\dfrac{\partial u}{\partial r}(x)w(x)\dfrac{1}{\phi^{N-1}(r)}\mathrm dV_g.
\end{equation}
Indeed, define the vector field $W=w(x)\partial_r/\phi^{N-1}(r)$ and note that
\begin{equation*}
    \mathrm{div}(W)=\dfrac{1}{\sqrt{|g|}}\dfrac{\partial}{\partial r}\left(\dfrac{w(x)}{\phi^{N-1}(r)}\sqrt{|g|}\right)=\dfrac{\partial w}{\partial r}(x)\dfrac{1}{\phi^{N-1}(r)},
\end{equation*}
where we used that $|g|=\phi^{2(N-1)}(r)|g_{\mathbb S^{N-1}}|$. By the divergence theorem, we have
\begin{align*}
\int_{M}u\dfrac{\partial w}{\partial r}\dfrac{1}{\phi^{N-1}(r)}\mathrm dV_g&=\int_{M}u\mathrm{div}(W)\mathrm dV_g=-\int_{M}g\left(\nabla u,W\right)\mathrm dV_g\\
&=-\int_{M}\dfrac{\partial u}{\partial r}w\dfrac{1}{\phi^{N-1}(r)}\mathrm dV_g,
\end{align*}
which implies \eqref{eqdivtheo}.

Now let us fix $u\in W^{k,p}_{\mathrm{rad}}(M)$. Given $\varphi\in C^\infty_0(0,R)$ we define $\psi(x)=\varphi(r)$ for all $x=(r,\theta)\in M$. Then, by \eqref{layercake}, \eqref{eq456} and \eqref{eqdivtheo},
\begin{align*}
\int_0^Rw\varphi^{(k)}\mathrm dr&=\dfrac{1}{\omega_{N-1}}\int_{M}u\dfrac{\partial^k\psi}{\partial r^k}\dfrac{1}{\phi^{N-1}(r)}\mathrm dV_g=\dfrac{(-1)^k}{\omega_{N-1}}\int_{M}\dfrac{\partial^ku}{\partial r^k}\psi\dfrac{1}{\phi^{N-1}(r)}\mathrm dV_g\\
&=(-1)^k\int_0^R\dfrac{\partial^ku}{\partial r^k}(r,\theta_0)\varphi\ \mathrm dr,
\end{align*}
with $\theta_0\in \mathbb S^{N-1}$ fixed. Then \eqref{eqc456} follows since $(\nabla^ku)_{1\cdots 1}=\partial^ku/\partial r^k$.

Note that \eqref{d456} is a direct consequence of $g^{ij}=0$ for $i\neq j$ in the following way
\begin{align*}
|\nabla^ku|_g^2&=\sum_{i_1,j_1,\ldots,i_k,j_k=1}^Ng^{i_1j_1}\cdots g^{i_kj_k}\left(\nabla^ku\right)_{i_1\cdots i_k}\left(\nabla^ku\right)_{j_1\cdots j_k}\\
&=\sum_{i_1,\ldots,i_k=1}^Ng^{i_1i_1}\cdots g^{i_ki_k}\left[\left(\nabla^ku\right)_{i_1\cdots i_k}\right]^2\\
&\geq\left[\left(\nabla^ku\right)_{1\cdots 1}\right]^2=|v^{(k)}|^2.
\end{align*}

Therefore, from \eqref{layercake}, \eqref{eqc456}, and \eqref{d456} we obtain that $v\in W^{k,p}((0,R),\phi^{N-1})$ with
\begin{equation*}
\|v\|_{W^{k,p}_{\phi^{N-1}}}^p=\sum_{j=0}^k\int_0^R|v^{(j)}|^p\phi^{N-1}(t)\mathrm dt\leq \dfrac{1}{\omega_{N-1}}\sum_{j=0}^k\int_{M}|\nabla^ju|^p\mathrm dV_g=\dfrac1{\omega_{N-1}}\|u\|^p_{W^{k,p}(M)}.
\end{equation*}
This concludes the proof.
\end{proof}

\section{Relations between \texorpdfstring{$W^{k,p}_{\mathrm{rad}}(M)$}{} and \texorpdfstring{$W^{k,p}((0,R),\phi^{N-1})$}{}}\label{sec4}

The main goal of this section is to prove Theorem~\ref{theo10}, which characterizes precisely when the condition $u\in W^{k,p}_{\mathrm{rad}}(M)$ is equivalent to $v\in W^{k,p}((0,R),\phi^{N-1})$, and conversely. The proof is divided into three steps, formulated as Propositions~\ref{propa}, \ref{propb}, and~\ref{propc}.

\begin{prop}\label{propa}
For each $p\geq1$ and $k\geq1$ integer, we have
\begin{enumerate}
    \item[(a)] $W^{k,p}_{\mathrm{rad}}(M)\hookrightarrow W^{k,p}((0,R),\phi^{N-1})$;
    \item[(b)] $W^{1,p}_{\mathrm{rad}}(M)\equiv W^{1,p}((0,R),\phi^{N-1})$.
\end{enumerate}
\end{prop}

\begin{proof}
\textit{(a)} Using Theorem \ref{theouv} and \eqref{layercake}, we have
\begin{align*}
\|\nabla^ju\|^p_{L^p(B_R)}&=\int_{B_R}|\nabla^ju(x)|_g^p\mathrm dV_g\leq \omega_{N-1}\int_0^R|v^{(j)}(t)|^p\phi^{N-1}(t)\mathrm dt\\
&=\omega_{N-1}\|v^{(j)}\|^p_{L^p_{\phi^{N-1}}},\quad\forall j=0,\ldots,k.
\end{align*}

\noindent \textit{(b)} It similarly follows as the item \textit{(a)}, but using that in Theorem \ref{theouv} the equality holds for $k=0$ and $k=1$. The equality $|\nabla^0u(x)|_g=|v^{(0)}(r)|$ is trivial and the case $k=1$ follows from
\begin{equation*}
\left(\nabla u\right)_i=\dfrac{\partial u}{\partial\theta_i}=\delta_{i1}v'(r)
\end{equation*}
and
\begin{equation*}
\left|\nabla u\right|_g^2=\sum_{i,j=1}^Ng^{ij}\left(\nabla u\right)_i\left(\nabla u\right)_j=|v'(r)|^2.
\end{equation*}
\end{proof}

In the next proposition, we impose some technical assumptions on $\phi$ and make use of a Hardy-type inequality (Proposition \ref{prophardy}) that will be established in the following section.

\begin{prop}\label{propb}
Let $R\in(0,\infty)$, $k\geq2$, and $p\geq1$. Assume that $|\phi'(r)|$ and $|\frac{\mathrm d^i}{\mathrm dr^i}(\phi''(r)\phi(r))|$ are bounded on $r\in(0,R)$ for every $i=0,\ldots,k-3$. Moreover, suppose that
\begin{equation*}
0<\liminf_{r\to R}\phi(r)\leq\limsup_{r\to R}\phi(r)<\infty.   
\end{equation*}
If $N>(k-1)p$, then $W^{k,p}_{\mathrm{rad}}(B_R)\equiv W^{k,p}((0,R),\phi^{N-1})$.
\end{prop}
\begin{proof}
By item \textit{(a)} of Proposition \ref{propa}, one of the desired continuous embeddings is already established. It remains to prove the converse embedding,
\begin{equation*}
W^{k,p}((0,R),\phi^{N-1})\hookrightarrow W^{k,p}_{\mathrm{rad}}(M).
\end{equation*}
Fix $v\in W^{k,p}((0,R),\phi^{N-1})$ and define $u(x)=v(r)$. We will prove, by induction on $k$, that $(\nabla^ku)_{i_1\ldots i_k}$ can be expressed as a linear combination of terms of the form
\begin{equation}\label{termsnablau}
\dfrac{\left(\phi'(r)\right)^{\ell_k}\Phi_{\gamma_k}(r)PD^{m_k}(\widetilde g)v^{(j_k)}(r)}{\phi^{n_k}(r)},
\end{equation}
where the indices satisfy
\begin{equation*}
1\leq j_k\leq k,\  0\leq\ell_k\leq k-1,\ 0\leq m_k\leq k-2,\  1-k\leq n_k\leq k-j_k-m_k,
\end{equation*}
and where:
\begin{itemize}
\item $\Phi_{\gamma_k}(r)$ is a product of powers of derivatives of $\phi''(r)\phi(r)$;
\item $PD^{m_k}(\widetilde g)$ is a product of derivatives of $\widetilde g$ up to order $m_k$.
\end{itemize}
More explicitly,
\begin{equation*}
\Phi_{\gamma_k}(r)=\left(\phi''(r)\phi(r)\right)^{\gamma_{k0}}\left[\left(\phi''(r)\phi(r)\right)'\right]^{\gamma_{k1}}\cdots\left[\left(\phi''(r)\phi(r)\right)^{(k-3)}\right]^{\gamma_{kk-3}},
\end{equation*}
where $\gamma_k=(\gamma_{k0},\ldots,\gamma_{kk-3})\in(\mathbb N\cup\{0\})^{k-2}$ is a multi-index such that $|\gamma_k|:=\gamma_{k0}+\cdots+\gamma_{kk-3}\leq k-2$. Moreover,
\begin{equation}\label{PDmk}
PD^{m_k}(\widetilde g)=D^{\beta_1}\left(\widetilde g_{a_1b_1}\right)\cdots D^{\beta_n}\left(\widetilde g_{a_nb_n}\right),
\end{equation}
where $\beta_1,\ldots,\beta_n$ are multi-indices whose entries, as well as $a_1,b_1,\ldots,a_n,b_n$, all belong to $\{i_1,\ldots,i_k\}\cap\{2,\ldots,N\}$, and $m_k:=|\beta_1|+\cdots+|\beta_n|$ is the total number of partial derivatives on $PD^{m_k}(\widetilde g)$.

We proceed by induction on $k$. For $k=2$, the claim is immediate since
\begin{equation*}
(\nabla^2 u)_{i_1i_2}=\delta_{1i_1}\delta_{1i_2}v''(r)+(1-\delta_{1i_1})(1-\delta_{1i_2})\phi(r)\phi'(r)\widetilde g_{i_1i_2}v'(r).
\end{equation*}
Assume the statement holds for some $k$; we prove it for $k+1$. Recall the recursive formula for the covariant derivative:
\begin{equation*}
    (\nabla^{k+1}u)_{i_1\ldots i_{k+1}}=\dfrac{\partial (\nabla^{k}u)_{i_2\ldots i_{k+1}}}{\partial \theta_{i_1}}-\sum_{\ell=2}^{k+1}\sum_{\alpha=1}^N\Gamma^\alpha_{i_1i_\ell}\left(\nabla^k u\right)_{i_2\cdots i_{\ell-1}\alpha i_{\ell+1}\cdots i_{k+1}.}
\end{equation*}
By the induction hypothesis, $(\nabla^ku)_{i_1\ldots i_{k}}$ is a linear combination of terms of the type \eqref{termsnablau}. Thus, the recursive formula yields contributions in $(\nabla^{k+1}u)_{i_1\ldots i_{k+1}}$ of the form
\begin{equation}\label{dtn}
\dfrac{\partial}{\partial \theta_{i_1}}\left[\dfrac{\left(\phi'(r)\right)^{\ell_k}\Phi_{\gamma_k}(r)PD^{m_k}(\widetilde g)v^{(j_k)}(r)}{\phi^{n_k}(r)}\right]
\end{equation}
and
\begin{equation}\label{ctn}
\Gamma_{i_1i_\ell}^\alpha\dfrac{\left(\phi'(r)\right)^{\ell_k}\Phi_{\gamma_k}(r)PD^{m_k}(\widetilde g)v^{(j_k)}(r)}{\phi^{n_k}(r)}.
\end{equation}
Using \eqref{ctn} together with the explicit formulas for the Christoffel symbols \eqref{christoffel}, we obtain, up to a constant, terms of the following three different forms
\begin{equation*}
\dfrac{\left(\phi'(r)\right)^{\ell_k}\Phi_{\gamma_k}(r)\left[\frac{\partial}{\partial\theta_a}\widetilde g_{bc}PD^{m_k}(\widetilde g)\right]v^{(j_k)}(r)}{\phi^{n_k}(r)},
\end{equation*}
where we used above that $\widetilde\Gamma_{ij}^k$ can be written as a sum of terms of the form $\frac{\partial}{\partial\theta_a}\widetilde g_{bc}$,
\begin{equation*}
\dfrac{\left(\phi'(r)\right)^{\ell_k+1}\Phi_{\gamma_k}(r)PD^{m_k}(\widetilde g)v^{(j_k)}(r)}{\phi^{n_k+1}(r)},
\end{equation*}
and
\begin{equation*}
\dfrac{\left(\phi'(r)\right)^{\ell_k+1}\Phi_{\gamma_k}(r)\left[\widetilde g_{ab}PD^{m_k}(\widetilde g)\right]v^{(j_k)}(r)}{\phi^{n_k-1}(r)}
\end{equation*}
appear on the expression of $(\nabla^{k+1}u)_{i_1\cdots i_{k+1}}$. Hence, to complete the induction step, it remains to check that the derivative term \eqref{dtn} has the desired structure given by \eqref{termsnablau}.

\smallskip

\noindent\textbf{Case 1:} $i_1\geq2$.\\
Note that, in this case,
\begin{equation*}
\eqref{dtn}=\dfrac{\left(\phi'(r)\right)^{\ell_k}\Phi_{\gamma_k}(r)\frac{\partial}{\partial\theta_{i_1}}\left(PD^{m_k}(\widetilde g)\right)v^{(j_k)}(r)}{\phi^{n_k}(r)}.
\end{equation*}
By the product rule, this is a linear combination of terms of the form
\begin{equation*}
\dfrac{\left(\phi'(r)\right)^{\ell_k}\Phi_{\gamma_k}(r)PD^{m_k+1}(\widetilde g)v^{(j_k)}(r)}{\phi^{n_k}(r)}
\end{equation*}
with $n_k\leq k-j_k-m_k=(k+1)-j_k-(m_k+1)$. Thus, the claim holds when $i_1\geq2$.

\smallskip

\noindent\textbf{Case 2:} $i_1=1$.\\
In this case,
\begin{align*}
\eqref{dtn}&=\ell_k\dfrac{\left(\phi'(r)\right)^{\ell_k-1}\left[\phi''(r)\phi(r)\Phi_{\gamma_k}(r)\right]PD^{m_k}(\widetilde g)v^{(j_k)}(r)}{\phi^{n_k+1}(r)}\\
&\quad+\dfrac{\left(\phi'(r)\right)^{\ell_k}\frac{\mathrm d}{\mathrm dr}\left(\Phi_{\gamma_k}(r)\right)PD^{m_k}(\widetilde g)v^{(j_k)}(r)}{\phi^{n_k}(r)}\\
&\quad+\dfrac{\left(\phi'(r)\right)^{\ell_k}\Phi_{\gamma_k}(r)PD^{m_k}(\widetilde g)v^{(j_k+1)}(r)}{\phi^{n_k}(r)}\\
&\quad-n_k\dfrac{\left(\phi'(r)\right)^{\ell_k+1}\Phi_{\gamma_k}(r)PD^{m_k}(\widetilde g)v^{(j_k)}(r)}{\phi^{n_k+1}(r)}
\end{align*}
The first, third, and last terms already have the desired form. Lastly, the second term can again be expanded by the product rule into a linear combination of factors of the desired type. Hence, the claim also holds for $i_1=1$. This completes the induction and proves that $(\nabla^ku)_{i_1\ldots i_k}$ can be written, up to constants, as a sum of terms of the form \eqref{termsnablau}.

By our assumptions,
\begin{equation*}
\left|\left(\phi'(r)\right)^{\ell_k}\Phi_{\gamma_k}(r)PD^{m_k}(\widetilde g)\right|\leq C,\quad\forall r\in(0,R).
\end{equation*}
Combined with $n_k\leq k-j_k-m_k$, \eqref{termsnablau}, and $\phi\in L^\infty(0,R)$, this implies
\begin{equation}\label{osrs}
\left|(\nabla^ku)_{i_1\ldots i_k}\right|\leq C\sum_{j=1}^k\sum_{m_k}\dfrac{|v^{(j)}(r)|}{\phi^{k-j-m_k}(r)}.
\end{equation}
Since $m_k$ represents the number of partial derivatives of $PD^{m_k}(\widetilde g)$ acting on $\widetilde g$, the worst estimate happens when $m_k$ is maximal, i.e. $m_k=I$, where 
\begin{equation*}
I=I(i_1,\ldots,i_k)=\#\{i_\ell\colon i_\ell\geq2\}
\end{equation*}
Therefore, from \eqref{osrs},
\begin{equation*}
\left|(\nabla^ku)_{i_1\ldots i_k}\right|\leq C\sum_{j=1}^k\dfrac{|v^{(j)}(r)|}{\phi^{k-j-I}(r)}.
\end{equation*}
Finally, using the formula \eqref{normkcov}, we obtain
\begin{align*}
|\nabla^ku|^2_g&=\sum_{i_1\ldots i_kj_1\ldots j_k=1}^Ng^{i_1j_1}\cdots g^{i_kj_k}(\nabla^ku)_{i_1\ldots i_k}(\nabla^ku)_{j_1\ldots j_k}=\sum_{i_1\cdots i_k=1}^Ng^{i_1i_1}\cdots g^{i_ki_k}(\nabla^ku)^2_{i_1\ldots i_k}\\
&\leq C\sum_{i_1\cdots i_k=1}^N\phi^{-2I}(r)\left(\sum_{j=1}^k\dfrac{|v^{(j)}(r)|}{\phi^{k-j-I}(r)}\right)^2=C\left(\sum_{j=1}^k\dfrac{|v^{(j)}(r)|}{\phi^{k-j}(r)}\right)^2.
\end{align*}
Applying the Hardy inequality from Proposition \ref{prophardy} (valid since $N>(k-1)p$), we conclude that
\begin{align*}
\int_M|\nabla^ku|_g^p\mathrm dV_g\leq C\sum_{j=1}^k\int_0^R\left|\dfrac{v^{(j)}(r)}{\phi^{k-j}(r)}\right|^p\phi^{N-1}(r)\mathrm dr\leq C\|v\|^p_{W^{k,p}_{\phi^{N-1}}}.
\end{align*}
This completes the proof of the proposition.
\end{proof}

The following proposition, under the assumption $N\leq(k-1)p$, exhibits an example of a radial function on $M$ that does not belong to the Sobolev space $W^{k,p}_{\mathrm{rad}}(M)$, while its radial representative does belong to the weighted Sobolev space $W^{k,p}((0,R),\phi^{N-1})$.

\begin{prop}\label{propc}
Assume that $R\in(0,\infty)$, $k\geq2$, $\phi\in L^{N-1}(0,R)$, and $N\leq (k-1)p$. Define the function $u\colon M\to\mathbb R$ by
\begin{equation*}
u(x)=r,\quad\forall x=(r,\theta)\in [0,R)\times\mathbb S^{N-1}.
\end{equation*}
Then $u\notin W^{k,p}_{\mathrm{rad}}(M)$, while $v\in W^{k,p}((0,R),\phi^{N-1})$. In particular,
\begin{equation*}
W^{k,p}_{\mathrm{rad}}(M)\not\equiv W^{k,p}((0,R),\phi^{N-1})\mbox{ if }N\leq(k-1)p.
\end{equation*}
\end{prop}
\begin{proof}
It suffices to show that $u\notin W^{k,p}_{\mathrm{rad}}(M)$, because $v\in W^{k,p}((0,R),\phi^{N-1})$ follows immediately from $v(t)=t$ and the assumption $\phi\in L^{N-1}(0,R)$.

Arguing as in the proof of the previous proposition,one can show, by induction on $k$, that each component $(\nabla^ku)_{i_1\ldots i_k}$ can be written as a linear combination of terms of the form
\begin{equation}\label{tnu}
\dfrac{\left(\phi'(r)\right)^{\ell_k}\Psi_{\gamma_k}(r)PD^{m_k}(\widetilde g)}{\phi^{n_k}(r)},
\end{equation}
where the indices satisfy
\begin{equation*}
0\leq\ell_k\leq k-1,\ 0\leq m_k\leq k-2,\ 1-k\leq n_k\leq k-3,
\end{equation*}
Here $PD^{m_k}(\widetilde g)$ is defined as in \eqref{PDmk} and
\begin{equation*}
\Psi_{\gamma_k}(r)=\left(\phi''(r)\right)^{\gamma_{k2}}\cdots\left(\phi^{(k-1)}(r)\right)^{\gamma_{kk-1}},
\end{equation*}
where $\gamma_k=(\gamma_{k2},\ldots,\gamma_{kk-1})\in\left(\mathbb N\cup\{0\}\right)^{k-2}$ is a multi-index such that $|\gamma_k|:=|\gamma_{k2}|+\cdots+|\gamma_{kk-1}|\leq k-2$. We omit the proof of this claim, as it is simpler and analogous to the proof of \eqref{termsnablau}. We only note that each term of the form \eqref{tnu} appearing in $(\nabla^ku)_{i_1\ldots i_k}$ gives rise, in $(\nabla^{k+1}u)_{i_1\ldots i_{k+1}}$, to the following four types of terms arising from the derivative term $\frac{\partial}{\partial\theta_a}\left(\nabla^ku\right)_{i_1\ldots i_k}$:
\begin{gather*}
    \dfrac{\left(\phi'(r)\right)^{\ell_k}\Psi_{\gamma_k}(r)\frac{\partial}{\partial\theta_a}\left(PD^{m_k}(\widetilde g)\right)}{\phi^{n_k}(r)},\\
    \dfrac{\left(\phi'(r)\right)^{\ell_k-1}\phi''(r)\Psi_{\gamma_k}(r)PD^{m_k}(\widetilde g)}{\phi^{n_k}(r)},\\
    \dfrac{\left(\phi'(r)\right)^{\ell_k}\frac{\mathrm d}{\mathrm dr}\Psi_{\gamma_k}(r)PD^{m_k}(\widetilde g)}{\phi^{n_k}(r)},\\
    \dfrac{\left(\phi'(r)\right)^{\ell_k+1}\Psi_{\gamma_k}(r)PD^{m_k}(\widetilde g)}{\phi^{n_k+1}(r)},
\end{gather*}
and to the following three types of terms arising from the Christoffel symbol contribution $\Gamma_{ij}^\ell\left(\nabla^ku\right)_{i_1\ldots i_k}$:
\begin{gather*}
\dfrac{\left(\phi'(r)\right)^{\ell_k}\Psi_{\gamma_k}(r)\left[\frac{\partial}{\partial\theta_a}\widetilde g_{bc}PD^{m_k}(\widetilde g)\right]}{\phi^{n_k}(r)},\\
\dfrac{\left(\phi'(r)\right)^{\ell_k+1}\Psi_{\gamma_k}(r)PD^{m_k}(\widetilde g)}{\phi^{n_k+1}(r)},\\
\dfrac{\left(\phi'(r)\right)^{\ell_k+1}\Psi_{\gamma_k}(r)PD^{m_k}(\widetilde g)}{\phi^{n_k-1}(r)}.
\end{gather*}
Hence, the growth of $(\nabla^{k+1}u)_{i_1\ldots i_{k+1}}$ is of order $O\left(\phi^{-n_k-1}(r)\right)$.

With this observation, we claim that
\begin{equation}\label{phi+o}
\left(\nabla^ku\right)_{1\ldots1ij}=(-1)^k(k-2)!\dfrac{\left(\phi'(r)\right)^{k-1}}{\phi^{k-3}(r)}\widetilde g_{ij}+o\left(\phi^{3-k}(r)\right),\quad\forall i,j=2,\ldots,N,
\end{equation}
as $r\to0$. Here the indices $1\ldots1ij$ indicate that all indices are equal to 1 except the last two. 

We prove \eqref{phi+o} by induction on $k$. For $k=2$, the claim follows directly from
\begin{equation*}
\left(\nabla^2u\right)_{ij}=\phi'(r)\phi(r)\widetilde g_{ij},\quad\forall i,j=2,\ldots,N.
\end{equation*}
Assume that \eqref{phi+o} holds for some $k\geq2$. Using the recurssive formula of the covariant derivative, we obtain
\begin{align*}
\left(\nabla^{k+1}u\right)_{1\ldots1ij}&=\dfrac{\partial}{\partial\theta_1}\left(\nabla^ku\right)_{1\ldots1ij}-\sum_{\alpha=2}^N\Gamma_{1i}^\alpha\left(\nabla^ku\right)_{1\ldots1\alpha j}-\sum_{\alpha=2}^N\Gamma_{1j}^\alpha\left(\nabla^ku\right)_{1\ldots1i\alpha}\\
&=\dfrac{\partial}{\partial r}\left(\nabla^ku\right)_{1\ldots1ij}-\left(\Gamma_{1i}^i+\Gamma^j_{1j}\right)\left(\nabla^ku\right)_{1\ldots1ij}.
\end{align*}
By the observation regarding the proof of \eqref{tnu}, all resulting terms is of order $o\left(\phi^{2-k}(r)\right)$, except those involving $(\phi'(r))^k\widetilde g_{ij}/\phi^{k-2}(r)$. Thus,
\begin{align*}
\left(\nabla^{k+1}u\right)_{1\ldots1ij}&=(-1)^k(k-2)!(3-k)\dfrac{\left(\phi'(r)\right)^k}{\phi^{k-2}(r)}\widetilde g_{ij}-2(-1)^k(k-2)!\dfrac{\left(\phi'(r)\right)^{k}}{\phi^{k-2}(r)}\widetilde g_{ij}+o\left(\phi^{2-k}(r)\right)\\
&=(-1)^{k+1}(k-1)!\dfrac{\left(\phi'(r)\right)^k}{\phi^{k-2}(r)}\widetilde g_{ij}+o\left(\phi^{2-k}(r)\right).
\end{align*}
This completes the induction proof of \eqref{phi+o}.

As a consequence of \eqref{phi+o}, there exists $R_0\in(0,R)$ small such that
\begin{equation*}
\left|\left(\nabla^ku\right)_{1\ldots122}\right|\geq(k-2)!\dfrac{\left(\phi'(r)\right)^{k-1}}{\phi^{k-3}(r)}\widetilde g_{22},\quad\forall x=(r,\theta)\in(0,R_0)\times\mathbb S^{N-1}.
\end{equation*}
Since $\phi'(0)=1$, possibly reducing $R_0$, we may assume
\begin{equation*}
\left|\left(\nabla^ku\right)_{1\ldots122}\right|\geq\dfrac{(k-2)!}{2^{k-1}}\dfrac{\widetilde g_{22}}{\phi^{k-3}(r)},\quad\forall x=(r,\theta)\in(0,R_0)\times\mathbb S^{N-1}.
\end{equation*}
Therefore, for all $x\in B_{R_0}(o)$,
\begin{equation*}
\left|\nabla^ku\right|^2_g\geq g^{11}\cdots g^{11}g^{22}g^{22}\left(\nabla^ku\right)^2_{1\ldots122}\geq\left[\dfrac{(k-2)!}{2^{k-1}\phi^{k-1}(r)}\right]^2.
\end{equation*}
Consequently,
\begin{equation*}
\int_M|\nabla^ku|_g^p\mathrm dV_g\geq\int_{B_{R_0}(o)}|\nabla^ku|_g^p\mathrm dV_g\geq\left(\dfrac{(k-2)!}{2^{k-1}}\right)^p\int_0^{R_0}\phi^{N-1-(k-1)p}(r)\mathrm dr.
\end{equation*}
Since $N\leq(k-1)p$, the integral diverges, and hence $u\notin W^{k,p}_{\mathrm{rad}}(M)$. This completes the proof.
\end{proof}
\begin{proof}[Proof of Theorem~\ref{theo10}]
The theorem follows directly from Propositions~\ref{propa}, \ref{propb}, and~\ref{propc}.
\end{proof}

\section{Radial Lemmas}\label{sec5}

Throughout this section, we assume $R\in(0,\infty)$, unless stated otherwise. Our main objective is to establish radial lemmas for $u\in W^{k,p}_{\mathrm{rad}}(M)$ and $v\in W^{k,p}((0,R),\phi^{N-1})$. Under suitable assumptions on the weight function $\phi$, the weighted Sobolev space $W^{k,p}((0,R),\phi^{N-1})$ turns out to be equivalent to the Euclidean weighted Sobolev space, corresponding to $\phi(t)=t$. This equivalence allows us to transfer the known Euclidean radial lemmas, together with a Hardy-type inequality, to the general weighted setting. As a consequence, we obtain the main results of this section, namely Theorem~\ref{theo11} and Corollary~\ref{corcompact}, which guarantee the continuous and compact embeddings of $W^{k,p}_{\mathrm{rad}}(M)$ into $L^q_{\phi^\theta}(M)$. 

Since $\phi(0)=0$ and $\lim_{r\to0}\frac{\phi(r)}{r}=1$, and assuming that
\begin{equation}\label{phi4}
0<\liminf_{r\to R}\phi(r)\leq\limsup_{r\to R}\phi(r)<\infty,
\end{equation}
there exist $C_1,C_2>0$ such that
\begin{equation*}
C_1\leq\dfrac{\phi(t)}{t}\leq C_2,\quad\forall t\in(0,R).
\end{equation*}
Consequently, the equivalence of weighted Sobolev spaces holds:
\begin{equation}\label{equivWw}
W^{k,p}((0,R),\phi^{N-1})\equiv W^{k,p}((0,R),t^{N-1}).
\end{equation}
Since the study of radial lemmas for functions belonging to the space on the right-hand side of \eqref{equivWw} is well established (see \cite[Section 2]{MR2838041}), the corresponding results in our general weighted framework follow directly. In particular, the next lemma and the following propositions are immediate consequences of those known results.

\begin{prop}\label{prop31}
Let $v\in W^{k,p}((0,R),\phi^{N-1})$ with $1\leq p<\infty$ and $R\in(0,\infty]$. Assume that \eqref{phi4} holds. Then there exists $V\in AC_{loc}^{k-1}((0,R])$ such that
\begin{equation*}
    v=V \quad \text{a.e. on } \quad (0,R).
\end{equation*}
Moreover, $V^{(k)}$ (in the classical sense) exists a.e. on $(0,R)$ and $V^{(k)}$ is a measurable function and
\begin{equation*}
\int_0^R\left|V^{(j)}(t)\right|^p\phi(t)^{N-1}\mathrm dt<\infty,\quad\forall j=0,1,\ldots,k.
\end{equation*}
\end{prop}

\begin{prop}\label{proprls}
Assume that \eqref{phi4} holds. If $N>kp$, then there exists $C>0$ such that for all $v\in W^{k,p}((0,R),\phi^{N-1})$ it holds
\begin{equation}\label{rlsv}
|v(t)|\leq C\dfrac{1}{\phi(t)^{\frac{N-kp}{p}}}\|v\|_{W^{k,p}_{\phi^{N-1}}},\quad\forall t\in(0,R].
\end{equation}
\end{prop}

\begin{prop}\label{propworst}
Assume that \eqref{phi4} holds. If $N=kp$ and $p>1$, then there exists a constant $C>0$ such that for all $v\in W^{k,p}((0,R),\phi^{N-1})$ and $t\in(0,R]$ it holds
\begin{equation}\label{LR-A}
|v(t)|\leq C\left[\left(\log\frac{R}{t}\right)^{\frac{p-1}p}+1\right]\|v\|_{W^{k,p}_{\phi^{N-1}}},\quad\forall t\in(0,R].
\end{equation}
\end{prop}

\begin{prop}\label{asfjanfkslns}
Assume that \eqref{phi4} holds. If $N=kp$ and $p=1$, then $W^{k,p}((0,R),\phi^{N-1})\hookrightarrow C([0,R])$. In other words, $W^{k,1}((0,R),\phi^{k-1})\hookrightarrow C([0,R])$.
\end{prop}

\begin{prop}\label{prophardy}
Assume that \eqref{phi4} holds. Given $j=0,\ldots,k$ with $N>jp$, there exists $C_j>0$ such that for all $v\in W^{k,p}((0,R),\phi^{N-1})$,
\begin{equation}\label{tl3}
\int_0^R\left|\dfrac{v^{(k-j)}(s)}{\phi(s)^j}\right|^p\phi(s)^{N-1}\mathrm ds\leq C_j\sum_{i=k-j}^k\int_0^R|v^{(i)}(s)|^p\phi(s)^{N-1}\mathrm ds.
\end{equation}
In particular, assuming $N>kp$,
\begin{equation*}
\left\|\dfrac{v}{\phi^k}\right\|_{L^p_{\phi^{N-1}}(0,R)}\leq C_k\|v\|_{W^{k,p}_{\phi^{N-1}}}.
\end{equation*}
\end{prop}

With all the results above inherited from the Euclidean space study, we are able to prove our embeddings results concerning $R\in(0,\infty)$. More precisely, now we prove Theorem \ref{theo11} and Corollary \ref{corcompact}.

\begin{proof}[Proof of Theorem \ref{theo11}]
\textit{(1)} This is a direct consequence of Theorem \ref{theo10}, item (1), combined with Proposition \ref{prop31}.

\vspace{0.2cm}

\noindent \textit{(2)} Equation \eqref{rls} follows directly from the radial lemma \eqref{rlsv} and Theorem \ref{theo10}. Using \eqref{rls} with $\theta\geq0$, we have
\begin{equation}\label{ccsa}
|u|^{\frac{\theta p}{N-kp}}\phi^\theta(r)\leq C\|u\|^{\frac{\theta p}{N-kp}}_{W^{k,p}(M)},\quad\forall x=(r,\theta)\in(0,R)\times \mathbb S^{N-1}.
\end{equation}

By $\lim_{r\to R}\phi^{(j)}(r)\in(0,\infty)$ for every $j\geq0$, we can extend the metric $g$ to the boundary of $B_R(0)$, making $(\overline{B_R(0)},g)$ a compact Riemannian manifold with boundary. Therefore, we can apply the Sobolev embedding theorem for compact Riemannian manifold with boundary (see \cite[Theorem 2.30]{MR1636569}) to obtain the continuous embedding $W^{k,p}(M)\hookrightarrow L^{\frac{Np}{N-kp}}(M)$. Combining this embedding with \eqref{ccsa}, we conclude that
\begin{equation*}
\int_{M}|u|^{\frac{p(\theta+N)}{N-kp}}\phi^\theta(r)\mathrm dV_g\leq C\|u\|^{\frac{\theta p}{N-kp}}_{W^{k,p}(M)}\|u\|^{\frac{Np}{N-kp}}_{L^{\frac{Np}{N-kp}}(M)}\leq C\|u\|_{W^{k,p}(M)}^{\frac{p(\theta+N)}{N-kp}},
\end{equation*}
which implies the continuous embedding $W^{k,p}_{\mathrm{rad}}(M)\hookrightarrow L^{\frac{p(\theta+N)}{N-kp}}_{\phi^\theta}(M)$.

\vspace{0.2cm}

\noindent\textit{(3)} Equation \eqref{rlsc} follows from the radial lemma \eqref{LR-A} and Theorem \ref{theo10}. We claim that
\begin{equation}\label{claimsa}
\left(\log\frac{R}{t}\right)^{\frac{p-1}p}+1\leq C_\varepsilon t^{-\varepsilon},\quad\forall t\in(0,R],\varepsilon>0,
\end{equation}
for some $C_\varepsilon>0$ independent of $t$. To prove \eqref{claimsa}, observe that
\begin{equation*}
\mbox{\eqref{claimsa} holds}\Leftrightarrow\lim_{t\to0}\dfrac{\left(\log\frac{R}{t}\right)^{\frac{p-1}p}+1}{t^{-\varepsilon}}=0\quad\forall\varepsilon>0\Leftrightarrow\lim_{t\to0}\dfrac{\log R-\log t}{t^{-\varepsilon}}=0\quad\forall\varepsilon>0.
\end{equation*}
Applying L'H\^opital's rule, we conclude that \eqref{claimsa} holds.

Using the radial lemma \eqref{LR-A} and \eqref{claimsa}, we obtain
\begin{equation*}
\int_M|u|^{q}\phi^\theta(r)\mathrm dV_g\leq\omega_{N-1}C_\varepsilon^q\int_0^Rt^{-\varepsilon q}\phi^\theta(t)\mathrm dt\leq C\int_0^Rt^{\frac{\theta-1}{2}}\mathrm dt<\infty,
\end{equation*}
where we choose $\varepsilon=(\theta+1)/2q$. This concludes the proof of \textit{(3)}.

\vspace{0.2cm}

\noindent\textit{(4)} This follows directly from Proposition \ref{asfjanfkslns} and Theorem \ref{theo10}.
\end{proof}

\begin{proof}[Proof of Corollary \ref{corcompact}]
Let $(u_n)$ be a bounded sequence in $W^{k,p}_{\mathrm{rad}}(M)$. Using the same argument in the proof of item \textit{(2)} of Theorem \ref{theo11}, we can consider $M$ as $(\overline{B_R(0)},g)$, a compact Riemannian manifolds with boundary. Since the compact embedding $W^{k,p}(M)\hookrightarrow L^1(M)$ holds for compact Riemannian manifold with boundary (see \cite[Theorem 2.34]{MR1636569}), it follows that, up to a subsequence, $(u_n)$ is Cauchy in $L^1(M)$. 

By interpolation inequality and the fact that $1\leq q<p^*_\theta=\frac{(\theta+N)p}{N-kp}$, there exists $\alpha\in(0,1]$ such that
\begin{align*}
\|u_n-u_m\|_{L^q_{\phi^\theta}(M)}&\leq\left\|\phi^{\frac{\theta}q}(r)u_n-\phi^{\frac{\theta}{q}}(r)u_m\right\|_{L^1(M)}^\alpha\left\|\phi^{\frac{\theta}q}(r)u_n-\phi^{\frac{\theta}{q}}(r)u_m\right\|_{L^{p^*_\theta}(M)}^{1-\alpha}.
\end{align*}
Diving each term by $\|\phi\|_{L^\infty}^{\frac{\theta}{q}}$ and using that $\left(\frac{\phi(r)}{\|\phi\|_{L^\infty}}\right)^{\frac{\theta}{q}}\leq\left(\frac{\phi(r)}{\|\phi\|_{L^\infty}}\right)^{\frac{\theta}{p^*_\theta}}\leq1$ for all $x\in M$, we obtain
\begin{align*}
\dfrac{\left\|u_n-u_m\right\|_{L^q_{\phi^\theta}(M)}}{\|\phi\|_{L^\infty}^{\frac{\theta}{q}}}&\leq\|u_n-u_m\|^\alpha_{L^1(M)}\dfrac{\Big\|\phi^{\frac{\theta}{p^*_\theta}}(r)u_n-\phi^{\frac{\theta}{p^*_\theta}}(r)u_m\Big\|^{1-\alpha}_{L^{p^*_\theta}(M)}}{\|\phi\|_{L^\infty}^{\frac{\theta}{p^*_\theta}}}\\
&\leq C\|u_n-u_m\|^\alpha_{L^1(M)}\|u_n-u_m\|^{1-\alpha}_{W^{k,p}(M)},
\end{align*}
where we used the continuous embedding $W^{k,p}_{\mathrm{rad}}(M)\hookrightarrow L^{p^*_\theta}_{\phi^\theta}(M)$ given in item \textit{(2)} of Theorem \ref{theo11}. Since $\alpha\in(0,1]$ and $(u_n)$ is Cauchy in $L^1(M)$, it follows that $(u_n)$ is also Cauchy in $L^q_{\phi^\theta}(M)$, and therefore, it converges in $L^q_{\phi^\theta}(M)$. This concludes the proof of the corollary.
\end{proof}

An analogous results holds for the weighted Sobolev space $W^{k,p}((0,R),\phi^{N-1})$. Since the proof follows the same lines as that of Theorem~\ref{theo11} and Corollary \ref{corcompact}, with only minor modifications, so we omit the details. Alternatively, by setting $\alpha_0=\cdots=\alpha_k=N-1$ in \cite[Theorem 1.1]{MR4790800} and considering the equivalence \eqref{equivWw}, we derive the following theorem.

\begin{theo}\label{theoweighted}
Let $R\in(0,\infty)$, $N\geq kp$, $\theta\geq N-kp-1$, $p\geq1$ real numbers, and $k\geq1$ integer. Assume that the limit $\lim_{r\to R}\phi(r)\in(0,\infty)$ exists.
\begin{flushleft}
$\mathrm{(1)}$ Every function $v\in W^{k,p}((0,R),\phi^{N-1})$ is almost everywhere equal to a function $V$ in $C^{k-1}((0,R])$. Additionally, all derivatives of $V$ of order $k$ (in the classical sense) exist almost everywhere for $t\in(0,R)$.\\
$\mathrm{(2)}$ If $N>kp$, then $W^{k,p}((0,R),\phi^{N-1})$ is continuously embedded in $L^q_{\phi^\theta}(0,R)$ for every $1\leq q\leq\frac{(\theta+1)p}{N-kp}$.\\
$\mathrm{(3)}$ If $N=kp$ and $p>1$, then $W^{k,p}((0,R),\phi^{N-1})$ is compactly embedded in $L^q_{\phi^\theta}(0,R)$ for all $1\leq q<\infty$.\\
$\mathrm{(4)}$ If $N=kp$ and $p=1$, then $W^{N,1}((0,R),\phi^{N-1})$ is continuously embedded in $C([0,R])$.
\end{flushleft}
\end{theo}

\begin{cor}
Assume the hypotheses of Theorem \ref{theoweighted} and suppose that $N>kp$ and $\theta\geq N-kp-1$. If $1\leq q<\frac{(\theta+1)p}{N-kp}$, then $W^{k,p}((0,R),\phi^{N-1})\hookrightarrow L^q_{\phi^\theta}(0,R)$ is compact.
\end{cor}

\section{Decay Lemma and its consequences}\label{sec6}

Now we focus on the case $R=\infty$. Since the domain of the exponential map $\exp_o$ is $B_R(0)=T_oM$, we have, by Hopf-Rinow's Theorem, that $(M,g)$ is complete. Using \cite[Theorem 3.1]{MR1688256}, we obtain that $C_0^\infty(M)$ is dense in $W^{1,p}(M)$. The following lemma concerns this property for the radial case.

\begin{lemma}\label{lemmadensity}
Suppose $R=\infty$. Then $C^\infty_{0,\mathrm{rad}}(M)$ the space of radial smooth functions with compact support is dense in $W^{1,p}_{\mathrm{rad}}(M)$ for all $1\leq p<\infty$.
\end{lemma}
\begin{proof}
By item (2) of Theorem \ref{theo10}, it is enough to show that for each $v\in W^{1,p}((0,\infty),\phi^{N-1})$ and $\varepsilon>0$, there exists a function $\phi\in C^\infty([0,\infty))$ such that $\mathrm{supp}(\phi)$ is compact and $\|v-\phi\|_{W^{1,p}_{\phi^{N-1}}}\leq\varepsilon$. For each $n\in\mathbb N$, define the cut-off function $\psi_n\colon \mathbb R\to \mathbb R$ given by
\begin{equation*}
\psi_n(r)=\left\{\begin{array}{lll}
     1,&\mbox{if }r\leq n,\\
     n+1-r,&\mbox{if }n<r\leq n+1,\\
     0,&\mbox{if }r\geq n+1.
\end{array}\right.
\end{equation*}
Note that $\mathrm{supp}(\psi_n)\subset (-\infty,n+1]$ and that $v_n:=v\psi_n\in W^{1,p}((0,\infty),\phi^{N-1})$ satisfies $v_n\to v$ in $W^{1,p}((0,\infty),\phi^{N-1})$. This allows us to fix $n_0\in\mathbb N$ such that
\begin{equation}\label{bs1}
\|v-v_{n_0}\|_{W^{1,p}_{\phi^{N-1}}}\leq\frac{\varepsilon}2.
\end{equation}

Given $T>0$, by a Hardy-type inequality (see \cite[Theorem 6.2]{zbMATH00046945}), there exists a constant $C_{T}>0$ such that
\begin{equation}\label{bs2}
\int_0^{T}|w-w(T)|^pt^{N-1}\mathrm dt\leq C_{T}\int_0^{T}|w'|^pt^{N-1}\mathrm dt,\quad\forall w\in AC_{\mathrm{loc}}(0,T),
\end{equation}
where $C_{T}>0$ satisfies the estimate
\begin{equation}\label{bs3}
C_{T}\leq p^{\frac{1}{p}}\left(\frac{p}{p-1}\right)^{\frac{p-1}p}\sup_{0<r<T}\left\|t^{\frac{N-1}p}\right\|_{L^p(0,r)}\left\|t^{-\frac{N-1}p}\right\|_{L^p(r,T)}.
\end{equation}
Fix $0<T_0\leq1$ sufficiently small such that
\begin{equation}\label{bs4}
\dfrac12t^{N-1}\leq \phi^{N-1}(t)\leq2t^{N-1},\quad\forall t\in[0,T_0],
\end{equation}
and
\begin{equation}\label{bs5}
\int_0^{T_0}|v'|^p\phi^{N-1}(t)\mathrm dt\leq \dfrac{\varepsilon^p}{2^{p+1}(4C_{\mathrm{max}}+1)},
\end{equation}
where $C_{\mathrm{max}}$ is the maximum of the right-hand term of \eqref{bs3}, which is bounded for $T\in(0,1]$ since it is nondecreasing in $T$. Since $W^{1,p}((T_0,n_0+1),\phi^{N-1})\equiv W^{1,p}(T_0,n_0+1)$, we can approximate $v_{n_0}$ by a smooth function in $W^{1,p}$. Moreover, we can impose that this smooth function can be smoothly extended to its value at $T_0$ for $t\in[0,T_0)$ and to its value at $n_0+1$ for $t\in(n_0+1,\infty)$. Specifically, there exists a smooth function $\phi\colon[0,\infty)\to\mathbb R$ such that $\phi(t)=v(T_0)$ on $[0,T_0]$, $\phi(t)=0$ on $[n_0+1,\infty)$, and 
\begin{equation}\label{bs6}
\|v_{n_0}-\phi\|_{W^{1,p}(T_0,n_0+1)}\leq \frac{\varepsilon}{2^{\frac{p+1}p}}.
\end{equation}

Using \eqref{bs2}, \eqref{bs4}, and \eqref{bs5}, we obtain
\begin{align}
\int_0^{T_0}|v-v(T_0)|^p\phi^{N-1}(t)\mathrm dt&\leq 2\int_0^{T_0}|v-v(T_0)|^pt^{N-1}\mathrm dt\leq 2C_{T_0}\int_0^{T_0}|v'|^pt^{N-1}\mathrm dt\nonumber\\
&\leq 4C_{T_0}\int_0^{T_0}|v'|^p\phi^{N-1}(t)\mathrm dt\leq 4C_{\mathrm{max}}\dfrac{\varepsilon^p}{2^{p+1}(4C_{\mathrm{max}}+1)}.\label{bs7}
\end{align}
Thus, by \eqref{bs1}, \eqref{bs6}, and \eqref{bs7}, we have
\begin{align*}
\|v-\phi\|_{W^{1,p}_{\phi^{N-1}}}&\leq \|v-v_{n_0}\|_{W^{1,p}_{\phi^{N-1}}}+\|v_{n_0}-\phi\|_{W^{1,p}_{\phi^{N-1}}}\\
&\leq\dfrac{\varepsilon}2+\left(\int_0^\infty|v_{n_0}-\phi|^p\phi^{N-1}(t)\mathrm dt+\int_0^\infty|v_{n_0}'-\phi'|^p\phi^{N-1}(t)\mathrm dt\right)^{\frac{1}{p}}\\
&\leq\dfrac{\varepsilon}2+\bigg(\int_{0}^{T_0}|v-v(T_0)|^p\phi^{N-1}(t)\mathrm dt+\int_{T_0}^{n_0+1}|v_{n_0}-\phi|^p\phi^{N-1}(t)\mathrm dt\\
&\quad+\int_0^{T_0}|v'|^p\phi^{N-1}(t)\mathrm dt+\int_{T_0}^{n_0+1}|v_{n_0}'-\phi'|^p\phi^{N-1}(t)\mathrm dt\bigg)^{\frac{1}{p}}\\
&\leq\dfrac{\varepsilon}2+\left(4C_{\mathrm{max}}\dfrac{\varepsilon^p}{2^{p+1}(4C_{\mathrm{max}}+1)}+\dfrac{\varepsilon^p}{2^{p+1}}+\dfrac{\varepsilon^p}{2^{p+1}(4C_{\mathrm{max}}+1)}\right)^{\frac1p}=\varepsilon.
\end{align*}
This completes the proof of the lemma.
\end{proof}

In this section, we consider the case $R=\infty$. Our objective is to develop a decay lemma (Lemma \ref{decaylemma}) for $u\in W^{1,p}_{\mathrm{rad}}(M)$ and establish conditions for the embedding $W^{k,p}_{\mathrm{rad}}(M)\hookrightarrow L^q_{\phi^\theta}(M)$. Following the same arguments as in the Euclidean scenario, developing an asymptotic decay for functions in $W^{1,p}_{\mathrm{rad}}(M)$ is crucial for proving these embeddings, which will be demonstrated in the proof of Theorem \ref{theo12}.

At the end of this section, we will present similar results for the weighted Sobolev space $W^{k,p}((0,\infty),\phi^{N-1}(t))$, noting that their proofs follow analogous arguments to those used for the Sobolev space $W^{k,p}_{\mathrm{rad}}(M)$.

\begin{lemma}\label{decaylemma}
Assume that $C_\phi>0$, where $C_\phi$ is given by \eqref{Cphi}. For each $u\in W^{1,p}_{\mathrm{rad}}(M)$, it holds
\begin{equation*}
|u(x)|\leq \left(\dfrac{p}{C_\phi^{N-1}\omega_{N-1}}\right)^{\frac1p}\|u\|_{L^p(M)}^{\frac{p-1}p}\|\nabla u\|_{L^p(M)}^{\frac1p}\phi^{\frac{1-N}{p}}(r),\quad\mbox{a.e. }x\in M.
\end{equation*}
\end{lemma}
\begin{proof}
By Lemma \ref{lemmadensity}, it is enough to prove the lemma for functions in $C^{\infty}_{0,\mathrm{rad}}(M)$. For now, assume $p>1$. Using that $u(x)=v(r)$ with $r=d(x)$, we have
\begin{equation*}
|u(x)|^p=|v(r)|^p=-\int_r^\infty(|v(t)|^p)'\mathrm dt=-p\int_r^\infty|v(t)|^{p-2}v(t)v'(t)\mathrm dt.
\end{equation*}
From the assumption, we obtain
\begin{align*}
|u(x)|^p&\leq p\int_r^\infty|v(t)|^{p-1}|v'(t)|\left(\frac{\phi(t)}{C_\phi\phi(r)}\right)^{N-1}\mathrm dt\\
&\leq \dfrac{p}{C_\phi^{N-1}}\|v\|_{L^p_{\phi^{N-1}}(0,\infty)}^{p-1}\|v'\|_{L^p_{\phi^{N-1}}(0,\infty)}\phi^{1-N}(t)\\
&=\frac{p}{C_\phi^{N-1}\omega_{N-1}}\|u\|_{L^p(M)}^{p-1}\|\nabla u\|_{L^p(M)}\phi^{1-N}(r).
\end{align*}
This concludes the case $p>1$.

The case $p=1$ immediately follows from
\begin{align*}
|u(x)|&=|v(r)|=\left|\int_r^\infty v'(t)\mathrm dt\right|\leq\int_r^\infty |v'(t)|\left(\frac{\phi(t)}{C_\phi\phi(r)}\right)^{N-1}\mathrm dt\\
&\leq \frac{1}{C_\phi^{N-1}\omega_{N-1}}\|\nabla u\|_{L^1(M)}\phi^{1-N}(r).
\end{align*}
Therefore, we conclude the lemma.
\end{proof}

Before proving Theorem \ref{theo12}, we need to establish the following proposition, which is the case $k=1$ of the theorem.

\begin{prop}\label{proppf12}
Let $\theta\geq0$ and $p\geq1$ real numbers. Assume that $C_\phi>0$.
\begin{flushleft}
    $\mathrm{(a)}$ If $N>p$, then the following continuous embedding holds:
    \begin{equation*}
    W^{1,p}_{\mathrm{rad}}(M)\hookrightarrow L^q_{\phi^\theta}(M)\quad\text{if}\quad p\leq q\leq p^*_\theta.
    \end{equation*}
    Moreover, the embedding is compact if $\lim_{r\to\infty}\phi(r)=\infty$, $q<p_\theta^*$, and $p\theta<q(N-1)$.
    $\mathrm{(b)}$ If $N=p$, then the following continuous embedding holds:
    \begin{equation*}
    W^{1,p}_{\mathrm{rad}}(M)\hookrightarrow L^q_{\phi^\theta}(M)\quad\text{if}\quad p\leq q<\infty.
    \end{equation*}
    Moreover, the embedding is compact if $\lim_{r\to\infty}\phi(r)=\infty$ and $p\theta<q(N-1)$.
\end{flushleft}
\end{prop}
\begin{proof}
    Let us first prove both the continuous embeddings. Fix $q\geq p$. By Theorem \ref{theo11}, we have $\|u\|_{L^q_{\phi^\theta}(B_R)}\leq C\|u\|_{W^{1,p}(M)}$. Thus, it remains to estimate the term $\|u\|_{L^q_{\phi^\theta}(M\backslash B_R)}$. Given $x\in M$ with $r\geq R$, applying Lemma \ref{decaylemma} and $C_\phi>0$, we obtain
    \begin{equation*}
|u(x)|^{q-p}\leq\dfrac{C}{\phi^{\frac{(N-1)(q-p)}{p}} (r)}\|u\|_{W^{1,p}(M)}^{q-p}\leq\dfrac{C}{\phi^{\frac{(N-1)(q-p)}{p}}(R)}\|u\|_{W^{1,p}(M)}^{q-p}.
    \end{equation*}
 Thus,
    \begin{align}
\int_{M\backslash B_R}|u|^q\phi^\theta(r)\mathrm dV_g&\leq \dfrac{C}{\phi^{\frac{(N-1)(q-p)}p}(R)}\|u\|^{q-p}_{W^{1,p}(M)}\int_{M\backslash B_R}|u|^{p}\phi^\theta(r)\mathrm dV_g\nonumber\\
&\leq \dfrac{C}{\phi^{\frac{(N-1)(q-p)}p+N-1-\theta}(R)}\|u\|^{q-p}_{W^{1,p}(M)}\int_{M\backslash B_R}|u|^{p}\phi^{N-1}(r)\mathrm dV_g.\label{rews}
    \end{align}
    This completes the proof of the continuous embeddings.

    Next, we prove the compact embeddings. We need to show that if $u_n\rightharpoonup0$ in $W^{1,p}_{\mathrm{rad}}(M)$, then $u_n\rightarrow0$ in $L^q_{\phi^\theta}(M)$. From our hypothesis for compact embeddings, we have $\frac{(N-1)(q-p)}p+N-1-\theta>0$. Then, by \eqref{rews}, $\lim_{r\to\infty}\phi(r)=\infty$, and $(u_n)$ is bounded in $W^{1,p}_{\mathrm{rad}}(M)$, given $\varepsilon>0$ there exists $\overline R>R_0$ sufficiently large such that
    \begin{equation*}
    \int_{M\backslash B_{\overline R}}|u_n|^q\phi^\theta(r)\mathrm dV_g<\dfrac{\varepsilon}2,\quad\forall n\in\mathbb N.
    \end{equation*}
    On the other hand, Corollary \ref{corcompact}  and the compact embedding of Theorem \ref{theo11} guarantees $n_0\in\mathbb N$ such that
    \begin{equation*}
        \int_{B_{\overline R}}|u_n|^q\phi^\theta(r)\mathrm dV_g<\dfrac{\varepsilon}2,\quad\forall n\geq n_0.
    \end{equation*}
    Therefore, $u_n\to0$ in $L^q_{\phi^\theta}(M)$ which concludes the proof of the proposition.
\end{proof}

We are now ready to prove our final main theorem, whose proof proceeds by induction on $k$ and relies on the previous proposition.

\begin{proof}[Proof of Theorem \ref{theo12}]
The proof for $k=1$ is the proof of Proposition \ref{proppf12}. Assume the theorem holds for $k-1$; we will prove it for $k$. Since $N>(k-1)p$ and $u,u'\in W^{k-1,p}_{\mathrm{rad}}(M)$, we have $u\in W^{1,q}_{\mathrm{rad}}(M)$ for $p\leq q\leq\frac{Np}{N-(k-1)p}=:\overline p$. In other words, we obtained that the following embedding is continuous
    \begin{equation}\label{eq58}
    W^{k,p}_{\mathrm{rad}}(M)\hookrightarrow W^{1,\overline p}_{\mathrm{rad}}(M).
    \end{equation}
    Since $\frac{(\theta+N)\overline p}{N-\overline p}=\frac{(\theta+N)p}{N-kp}$, applying the Proposition \ref{proppf12} we have that
    \begin{equation}\label{eq59}
    W^{1,\overline p}_{\mathrm{rad}}(M)\hookrightarrow L_{\phi^\theta}^q(M)
    \end{equation}
    where $p\leq q\leq p^*_\theta$ in Sobolev case and $p\leq q<\infty$ in Sobolev Limit case. \eqref{eq58} together with \eqref{eq59} conclude the continuous embeddings.

    Under the conditions (i) or (ii), we obtain that \eqref{eq59} is a compact embedding by Proposition \ref{proppf12} in both Sobolev and Sobolev Limit cases. Therefore, \eqref{eq58} and \eqref{eq59} imply the desired compact embedding. This concludes the proof of the theorem.
    \end{proof}

An analogous results holds for the weighted Sobolev space $W^{k,p}((0,\infty),\phi^{N-1})$. The argument is a straightforward adaptation of the proofs of Lemma \ref{decaylemma} and Theorem \ref{theo12}, requiring only minor adjustments; therefore, the proof is not reproduced here.

    \begin{lemma}
Assume that $C_\phi>0$. For each $v\in W^{1,p}((0,\infty),\phi^{N-1})$, it holds
\begin{equation*}
|v(t)|\leq \left(\dfrac{p}{C_\phi^{N-1}}\right)^{\frac1p}\|v\|_{L^p_{\phi^{N-1}}}^{\frac{p-1}p}\|v'\|_{L^p_{\phi^{N-1}}}^{\frac1p}\phi^{\frac{1-N}{p}}(t),\quad\mbox{a.e. } t\in(0,\infty).
\end{equation*}
\end{lemma}
\begin{proof}
It follows by employing a similar approach to the proof of Lemma \ref{decaylemma}.
\end{proof}

\begin{theo}\label{theo122}
Let $\theta\geq N-kp-1$ and $p\geq1$ real numbers, and $k\geq1$ integer. Assume that $C_\phi>0$.
\begin{flushleft}
$\mathrm{(1)}$ \justifying{Every function $v\in W^{k,p}((0,\infty),\phi^{N-1})$ is almost everywhere equal to a function $V$ in $C^{k-1}(0,\infty)$. In addition, all derivatives of $V$ of order $k$ (in the classical sense) exist almost everywhere for $t\in(0,\infty)$.}

\noindent$\mathrm{(2)}$ If $N>kp$, then the following continuous embedding holds:
\begin{equation*}
W^{k,p}((0,\infty),\phi^{N-1})\hookrightarrow L^q_{\phi^\theta}(0,\infty)\quad\text{if}\quad p\leq q\leq \dfrac{(\theta+1)p}{N-kp},
\end{equation*}
Moreover, the embedding is compact if $\lim_{r\to\infty}\phi(r)=\infty$, $q<\frac{(\theta+1)p}{N-kp}$, and $p\theta<q(N-1)$.
\noindent$\mathrm{(3)}$ \justifying{If $N=kp$, then the following continuous embedding holds:}
\begin{equation*}
W^{k,p}((0,\infty),\phi^{N-1})\hookrightarrow L^q_{\phi^\theta}(0,\infty)\quad\text{if}\quad p\leq q<\infty.
\end{equation*}
Moreover, the embedding is compact if $\lim_{r\to\infty}\phi(r)=\infty$ and $p\theta<q(N-1)$.
\end{flushleft}
\end{theo}
\begin{proof}
Using the same ideas as in the proof of Proposition \ref{proppf12} and Theorem \ref{theo12}, but using Theorem \ref{theoweighted} instead of Theorem \ref{theo11}.
\end{proof}

\section{Possible Further Research Directions}\label{sec7}

In this final section, we present some natural continuations of the present analysis to emphasize the versatility of our work.

\smallskip

\begin{center}
\textbf{Adams-type inequalities for spherically Symmetric Riemannian manifolds}
\end{center}

\smallskip

In the critical case $N=kp$, the Sobolev embedding $W^{k,p}_{\mathrm{rad}}(M)\hookrightarrow L^q_{\phi^\theta}(M)$ can be improved to an embedding that admits exponential growth. Motivated by the classical works \cite{zbMATH03337983,zbMATH03261965,zbMATH04099653}, our previous work \cite{MR4959099} established this refinement for the hyperbolic space. We can consider the corresponding results here. Define
\begin{equation*}
\mathcal{L}_{\mu,M}:= \sup_{\substack{u\in W^{k,p}_{\mathrm{rad}}(M)\\ \|u\|_{W^{k,p}(M)}\leq 1}}\int_{M}e^{\mu|u|^{\frac{p}{p-1}}}\phi(d(x))^\theta \mathrm dV_g
\end{equation*}
and
\begin{equation*}
\mathcal{L}_\mu:=\sup_{\substack{v\in W^{k,p}((0,R),\phi^{N-1})\\ \|v\|_{W^{k,p}_{\phi^{N-1}}}\leq1}}\int_0^Re^{\mu|v|^{\frac{p}{p-1}}}\phi(t)^\theta\mathrm dt.
\end{equation*}
No boundary conditions are imposed on the functions. The Sobolev norm includes all derivatives of orders 0 through $k$, which is different from the norm used by Adams in \cite{zbMATH04099653}. Since the proofs in the spherically symmetric Riemannian manifold setting are very similar to those for the hyperbolic space, we omit them.
\begin{theo}
Let $\theta>-1$ and $W^{k,p}((0,R),\phi^{N-1})$ be the weighted Sobolev space with $N=kp$, $0<R<\infty$, and $p>1$.
\begin{flushleft}
    $\mathrm{(a)}$ For all $\mu\geq0$ and $u\in W^{k,p}((0,R),\phi^{N-1})$, we have $\exp(\mu|u|^{\frac{p}{p-1}})\in L^1_{\phi^{\theta}}(0,R)$.\\
    $\mathrm{(b)}$ If $0\leq\mu<\mu_0$, then $\mathcal{L}_\mu$ is finite, where
    \begin{equation*}
        \mu_{0}:=(\theta+1)[(k-1)!]^{\frac{p}{p-1}}.
    \end{equation*}
    \justifying{Moreover, $\mathcal{L}_\mu$ is attained by a nonnegative function $u_0\in W^{k,p}((0,R),\phi^{N-1})$ with $\|u_0\|_{W^{k,p}_{\phi^{N-1}}}=1$.}\\
    $\mathrm{(c)}$ If $\mu>\mu_0$, then $\mathcal{L}_\mu=\infty$.
\end{flushleft}
\end{theo}
\begin{theo}
Let $\theta>-N$ and $W^{k,p}_{\mathrm{rad}}(M)$ be the radial Sobolev space with $N=kp$, $0<R<\infty$, and $p>1$.
\begin{flushleft}
$\mathrm{(a)}$ For all $\mu\geq0$ and $u\in W^{k,p}_{\mathrm{rad}}(M),$ we have $\exp(\mu|u|^{\frac{p}{p-1}})\in L^1_{\phi^\theta}(M)$.\\
$\mathrm{(b)}$ If $0\leq\mu<\mu_{0,M}$, then $\mathcal{L}_{\mu,M}$ is finite, where
\begin{equation*}
\mu_{0,M}:=(\theta+N)[\omega_{N-1}^{\frac{1}{p}}(k-1)!]^{\frac{p}{p-1}}.
\end{equation*}
\justifying{Moreover, assuming $\theta\geq0$, $\mathcal{L}_{\mu,M}$ is attained by a nonnegative function $u_0\in W^{k,p}_{\mathrm{rad}}(M)$ with $\|u_0\|_{W^{k,p}(M)}=1$.}\\
$\mathrm{(c)}$ If $\mu>C_0^{\frac{p}{p-1}}(\theta+N)[(k-1)!]^{\frac{p}{p-1}}$, where $C_0$ is the constant associated with the continuous embedding $W^{k,p}((0,R),\phi^{N-1})\hookrightarrow W^{k,p}_{\mathrm{rad}}(M)$ as provided in Theorem \ref{theo10}, then $\mathcal{L}_{\mu,M}=\infty$.
\end{flushleft}
\end{theo}

\smallskip

\begin{center}
\textbf{Radial Sobolev embeddings for general Riemannian manifolds}
\end{center}

\smallskip

Instead of assuming that $(M,g)$ is a spherically symmetric, one may consider a general $N$-dimensional Riemannian manifold and seek an analogue of Theorem \ref{theouv}. The goal is to show that embeddings of the form
\begin{equation}\label{eqs71}
W^{k,p}_{\mathrm{rad}}(M)\hookrightarrow W^{k,p}((0,R),\phi^{N-1})\mbox{ and }W^{1,p}_{\mathrm{rad}}(M)\equiv W^{1,p}((0,R),\phi^{N-1})
\end{equation}
continue to hold for a suitable weight $\phi$ determined by the geometry. Let $B(o,R)\subset M$ be a normal geodesic ball and let $u\colon B(o,R)\to\mathbb R$ be radial with respect to $o$, meaning that there exists $v\colon [0,R)\to\mathbb R$ such that $u(x)=v(d(x))$, where $d(x)=\mathrm{dist}_g(o,x)$. Here we do the calculations under the assumption that the functions are smooth. However, as demonstrated in the proof of Theorem \ref{theouv}, it is expected that the analogous equations hold for weak derivatives. The first derivatives satisfy
\begin{equation*}
\nabla u(x)=v'(d(x))\nabla d(x).
\end{equation*}
Using \eqref{layercakehard}, this yields the right-hand embedding in \eqref{eqs71} for $M=B(o,R)$ and
\begin{equation*}
\phi(t)=\int_{\mathbb S^{N-1}}\sqrt{\det(h_{ij}(t,\theta))}\mathrm d\theta\mathrm dt.
\end{equation*}
Under this setting of general Riemannian manifolds, we must assume $\phi$ is given by the expression above. For the second derivative, we have
\begin{equation*}
\nabla^2u(x)=v''(d(x))\nabla d\otimes\nabla d(x)+v'(d(x))\nabla^2d(x).
\end{equation*}
Then, by $|\nabla d\otimes\nabla d|_g=1$,
\begin{equation*}
|v''(d(x))|\leq |\nabla^2u(x)|_g+|\nabla^2d(x)|_g|v'(d(x))|.
\end{equation*}
Geometric bounds on $\nabla^2 d$ inside the normal ball then imply the left-hand embedding in \eqref{eqs71} for $k=2$. Iterating this argument, one expects that for higher orders
\begin{equation*}
|v^{(k)}(d(x))|\leq |\nabla^ku(x)|_g+\sum_{i=1}^{k-1}C_i|v^{i}(d(x))|,
\end{equation*}
where the constants $C_i$ depend only on suprema of derivatives of the distance function. This provides a recursive control of the derivatives of $v$ by covariant derivatives of $u$ on $B(o,R)$ ensures the valid (in the normal geodesic ball) of the left-hand embedding of \eqref{eqs71} for all $k$.

A complete study requires further steps. First, the above identities and inequalities must be justified in the weak sense. Second, one must identify suitables assumptions on $\phi$ ensuring the validity of the weighted Sobolev embedding and radial lemma estimates. Establishing these ingredients would extend the radial embedding results of the present paper to a broad geometric setting.

\smallskip

\begin{center}
\textbf{Fractional Sobolev framework}
\end{center}

\smallskip

Another natural direction concerns transporting the present radial theory to fractional Sobolev spaces. To our knowledge, the analysis of radial functions in $W^{s,p}$ on Riemannian manifolds appears to be largely unexplored until now. In the Euclidean setting, several results indicate that radial symmetry again leads to a reduction to one-dimensional spaces. For instance, in \cite[Theorem 1.1]{MR2917408} it is shown that if $s\in(0,1)$ and $u(x)=v(|x|)$ is a radial $C^2$ function defined on $\mathbb R^N$ satisfying
\begin{equation*}
\int_0^{\infty}\dfrac{|v(r)|}{(1+r)^{n+2s}}r^{n-1}\mathrm dr<\infty,
\end{equation*}
then the fractional Laplacian admits the representation
\begin{equation*}
(-\Delta)^su(x)=C_{s,N}r^{-2s}\int_1^\infty\left(v(r)-v(r\tau)+(v(r)-v(\frac{r}{\tau}))\tau^{-n+2s}\right)\tau(\tau^2-1)^{-1-2s}H(\tau)\mathrm d\tau,
\end{equation*}
where $r=|x|\in(0,\infty)$, $C_{s,N}$ is a positive normalization constant, and
\begin{equation*}
H(\tau)=2\pi\alpha_n\int_0^\pi\sin^{n-2}\theta\dfrac{\sqrt{\tau^2-\sin^2\theta}+\cos\theta)^{1+2s}}{\sqrt{\tau^2-\sin^2\theta}}\mathrm d\theta,\quad\tau\geq1,\quad\alpha_N=\dfrac{\pi^{\frac{n-3}2}}{\Gamma(\frac{n-1}2)}.
\end{equation*}
In addition, radial lemmas are available in fractional settings. For example, \cite[Theorem 3.1]{MR3717735} proves that if $u\in H^s_{q,a}(\mathbb R^N)$ is radial and the parameters $q,a,s,\sigma$ satisfy suitable relations, then
\begin{equation*}
|u(x)|\leq C|x|^{-\sigma}\|u\|_{H^s_{q,a}},\quad\forall x\in\mathbb R^N\backslash\{0\}.
\end{equation*}
Moreover, every such function has a representative continuous in $\mathbb R^N\backslash\{0\}$. Symmetrization techniques and fractional P\'olya-Szeg\"o type inequalities relevant to this context are discussed in \cite{MR4548858}.

A comprehensive theory in the Riemannian framework would require identifying how the fractional seminorm
\begin{equation*}
[u]^p_{W^{s,p}(M)}=\int\int_{M\times M}\dfrac{|u(x)-u(y)|^p}{d_g(x,y)^{N+sp}}\mathrm dV_g(x)\mathrm dV_g(y)
\end{equation*}
reduces under radial symmetry. One expects that, in geodesic polar coordinates, there exists an analogous radial weight to $\phi^{N-1}(t)$, while the kernel $d_g(x,y)^{-N-sp}$ induces a nonlocal one-dimensional term. The main objective would then be to establish an equivalence between the radial fractional Sobolev space $W^{s,p}_{\mathrm{rad}}(M)$ and a suitable weighted one-dimensional fractional Sobolev space that contains $v(t)$. Such a result should hold, at least, on spherically symmetric Riemannian manifolds and could be extended to more general manifolds under appropriate geometric assumptions.

\begin{center}
\textbf{Supercritical Sobolev-type inequality}
\end{center}

Let $(M,g)$ be a spherically symmetric Riemannian manifold of dimension $N>2k$. Assume that for every $j\geq0$ the limits $\lim_{r\to R}\phi^{(j)}(r)\in(0,\infty)$ exist. Consider the best constant for the supercritical Sobolev-type inequality given by
\begin{equation*}
U_{n,k,\alpha,\phi}=\sup\left\{\int_M|u(x)|^{p(x)}\mathrm dV_g\colon u\in W^{k,2}_{0,\mathrm{rad}}(M),\|\nabla^ku\|_{L^2(M)}\leq1\right\},
\end{equation*}
where the exponent has one of the forms
\begin{equation*}
p(x)=\frac{2N}{N-2k}+d(x)^\alpha\mbox{ or }p(x)=\frac{2N}{N-2k}+\phi(r)^\alpha.
\end{equation*}
In our opinion, the choice of $p(x)$ with $d(x)$ is more promising. The exponent exceeds the critical Sobolev value $\frac{2N}{N-2k}$, so the problem lies in the supercritical framework. In the Euclidean case, this supremum was studied in \cite{MR3514752} for $k=1$ and in \cite{MR4068994} for higher-order derivatives. In their setting, they showed $U_{n,k,\alpha,\phi}<\infty$ for every $\alpha>0$, thereby establishing a supercritical Sobolev-type inequality. Moreover, for $\alpha\in(0,N-2k]$, they obtained the existence of a nontrivial radial solution of the associated Euler--Lagrange equation
\begin{equation*}
\left\{\begin{array}{ll}
     (-\Delta)^ku=u^{p(x)-1}&\mbox{in }M,  \\
     u>0,&\mbox{in }M,\\
     \dfrac{\partial^j u}{\partial r}=0,&\mbox{on }\partial M,\quad\forall j=0,\ldots,k-1.
\end{array}\right.
\end{equation*}

The contributions of the present paper provide the necessary tools to improve these Euclidean results to spherically symmetric Riemannian manifolds. In particular, they allow us to establish the finiteness of $U_{N,k,\alpha,\phi}$ under suitable assumptions on $\phi$, and obtain the existence of radial solutions to the corresponding polyharmonic equation with variable supercritical exponent.

\bigskip 

 \noindent {\bf Funding:} J. M. do \'O acknowledges partial support from CNPq through grants\linebreak 312340/2021-4, 409764/2023-0, 443594/2023-6, CAPES MATH AMSUD grant 88887.878894/2023-00, and Para\'iba State Research Foundation (FAPESQ), grant no 3034/2021. G. Lu acknowledges partial support from Simons collaboration grants 519099 and 957892 from the Simons Foundation. R. Ponciano acknowledges partial support from São Paulo Research Foundation (FAPESP) grants 2023/07697-9 and 2025/07027-9.

\end{document}